\documentclass{siamltex}%
\usepackage{amsfonts}
\usepackage{float} 
\usepackage{hyperref}
\usepackage{amsmath}%
\usepackage{amssymb}%
\usepackage{graphicx}

\usepackage{subfig}
\providecommand{\U}[1]{\protect\rule{.1in}{.1in}}

\newtheorem{Proposition}[theorem]{Proposition}

\newtheorem{remark}[theorem]{Remark}

\begin{document}

\title{Stability and Accuracy analysis of the $\theta$ Method and 3-Point Time filter\thanks{
The research was partially supported by NSF grant DMS-2110379.}}
\author{
Nicholas Hurl\thanks{Department of Mathematics, Duquesne University, Pittsbugh, PA-15282 (hurln@duq.edu).}
\and Farjana Siddiqua\thanks{Department of Mathematics, University of Pittsburgh, Pittsburgh, PA-15260
  (fas41@pitt.edu ).}
\and Shuxian Xu\thanks{Department of Mathematics, University of Pittsburgh
  (shx34@pitt.edu).}
}
\maketitle

\begin{abstract}
This paper analyzes a $\theta$-method and 3-point time filter. This approach adds one additional line of code to the existing source code of $\theta$-method. We prove the method's $0$-stability, accuracy, and $A$-stability for both constant time step and variable time step. Some numerical tests are performed to validate the theoretical results. 

\end{abstract}


\section{Introduction}
Time filters have been studied to improve stability of the 
Crank-Nicolson-Leapfrog in \cite{hurl2014stability,asselin1972frequency} and later to improve the accuracy of the Backward Euler method in \cite{guzel2018time,williams2009proposed}. The accuracy of the Backward Euler method can be increased by a method introduced by Guzel and Layton in \cite{guzel2018time}. Herein we examine the effects on $A$-stability and accuracy of a general 3-point filter applied to the $\theta$-method. The main result of this paper is analyzing the stability and accuracy of adding a post-processing step to the $\theta$ method. 
It is found that we need to add one more line to the source code that works for Backward Euler, Trapezoidal, or Forward Euler method to either increase stability or numerical accuracy or both. 
Consider the initial value problem (IVP),
$$y'(t)=f(t, y(t)),\quad \text{for} \quad t>0 \quad \text{and} \quad y(0)=y_0. $$
Denote the $n$th time step size by $k_n$. Let $t_{n+1}=t_n+k_n,\ \tau=\frac{k_n}{k_{n-1}},\  \nu$ be an algorithm parameter and $y_n$ be an approximation to $y(t_n)$. Let $y_{n+1}^*$ and $y_{n+1}$ denote unfiltered and filtered values, respectively. Discretize this IVP using $\theta$ method followed by a simple filter which is shown below (for constant time step):
\begin{equation} \label{11}
\begin{split}
     \text{Step 1}: y_{n+1}^* &=y_n+k((1-\theta) f(t_n,y_n)+\theta f(t_{n+1},y_{n+1}^*)) ,\\
     \text{Step 2}: y_{n+1}&=y_{n+1}^*-\frac{\nu}{2}\{ y_{n+1}^*-2y_n+y_{n-1}\}.
     \end{split}
\end{equation}
The combination of $\theta$ method and a 3-point filter produces a consistent approximation and achieves second-order accuracy for $\nu = 2\frac{2\theta-1}{2\theta+1}$ (Proposition 3.1). The method \eqref{11} is 0-stable for $-2 \leq \nu < 2$ and A-stable for $\theta\geq\frac{1}{2}$ and $ 2-4\theta\leq (2\theta+1)\nu\leq 4\theta-2$ (Proposition 3.2). Since Step 2 with $\nu = 2\frac{2\theta-1}{2\theta+1}$ has greater accuracy than Step 1, we can have an estimator which is the difference between pre-filter and post-filter
$$EST=\| y_{n+1}-y_{n+1}^*\| .$$
In Section 4, variable time step case is considered and the steps are as follows:
\begin{equation} \label{12}
\begin{split}
     \text{Step 1}: y_{n+1}^* &=y_n+k_n((1-\theta) f(t_n,y_n)+k_n \theta f(t_{n+1},y_{n+1}^*)) ,\\
     \text{Step 2}: y_{n+1}&=y^*_{n+1}-\frac{\nu}{1+\tau} \left( y^*_{n+1}- (1+\tau) y_n+\tau y_{n-1}\right).
\end{split}
\end{equation}
For $ \nu = \frac{\tau(1+\tau)(2\theta-1)}{2\theta\tau+1}$, \eqref{12} is second-order convergent (Proposition \ref{VariableTimeStepConsistency}). Recently in \cite{qin2022analysis}, a $\theta$ scheme with a time filter has been implemented. But it only considered a constant time step and is developed for specific applications in the unsteady Stokes-Darcy model.
In our paper, we are considering the general method with both constant and variable timesteps.
Numerical tests in Section 5 confirm the theoretical prediction of good accuracy with an appropriate choice of $\nu$. It is observed that we get a balance between stability and accuracy. For example, as we pick $\nu$ near $2$, the stability region gets bigger, but as we pick $\nu$ near $2$, the LTE goes to infinity. When $\theta>\frac{1}{2}$, we always have $A$- stability, and in that case, we would choose $\nu$ to get second-order accuracy. If $\theta<\frac{1}{2}$, we would not have $A$-stability. In this case, we choose $\nu$ to increase $A_0$- stability or $A_{\frac{\pi}{4}}$- stability rather than accuracy.
 
\section{Notations and Preliminaries}
In this section, we provide fundamental mathematical definitions and theorems.
\begin{definition}(Local Truncation Error) 
Local truncation error (LTE), $\tau_n$,  at step $n$ computed from the difference between the left- and the right-hand side of the equation for the increment $y_n\approx y_{n-1} + h A(t_{n-1}, y_{n-1}, h,f)$, where $k_n = t_n-t_{n-1}$:
\begin{equation*}
    \tau_n = {y_n - y_{n-1}} - k_n A(t_{n-1}, y_{n-1}, h,f).
\end{equation*}
\end{definition}
\begin{definition}(Consistent) The difference method is consistent of order $p$ if $\tau = O(h_n^{p+1})$ for positive integer $p$.
\end{definition}
\begin{definition}(Order of Accuracy) The difference method has the order of accuracy $p$ if $\tau = O(h_n^{p+1})$ for positive integer $p$.
\end{definition}
\begin{definition}($0$-stability) A difference method is $0$-stability if there are positive constants $h_0$ and $K$ such that for any mesh function $x_h$ and $z_h$ with $h\leq h_0$,
\begin{equation*}
    |x_n-z_n|\leq K\{|x_0-z_0|+\max_{1\leq j\leq N} |\mathcal{N}_hx_h(t_j)-\mathcal{N}_hz_h(t_j)|\},\quad 1\leq n\leq N.
\end{equation*}
\end{definition}

\begin{theorem}(Dahlquist Equivalence Theorem)
A difference method is convergent if and only if it is consistent and stable.
\end{theorem}


The following lemma is found in Dahlquist 
\cite{dahlquist1979some} and summarizes $A$-stability for two-step methods with variable timesteps. Let 
\begin{align*}
    \rho(\eta)&=\alpha_2 \eta^2+\alpha_1 \eta+\alpha_0,\\
    \sigma(\eta)&=\beta_2\eta+\beta_1 \eta+\beta_0.
\end{align*}
be the characteristic polynomials of a 2-step method. 
\begin{lemma}(see page 4 Lemma 2) \label{A-stableLemma}
The consistent, A-stable two-step methods can be expressed in terms of three non-negative parameters, $a,b,c$
\begin{align*}
    &2\alpha_2=c+1, &4\beta_2=1+b+(a+c),\\
    &2\alpha_1=-2c, &4\beta_1=2(1-b),\\
    &2\alpha_0=c-1, &4\beta_0=1+b-(a+c).
\end{align*}
Conversely, $c=-\alpha_1$, $b=1-2\beta_1$, $a+c=2(\beta_2-\beta_0)$.
\end{lemma}

A simple check that $a,b,c\geq 0$ will be completed in sub-sections \ref{ConstantTimeStepStability} and \ref{VariableTimeStepStability} to show $A$-stability.
\section{Constant Time step} We consider the initial value problem
 $y'(t)=f(t,y(t)),$  for $t>0$ and $y(0)=y_0$. Denote the $n$-th time stepsize by $k_n$. Let $t_{n+1}=t_n+k_n$ and $y_n$ an approximation to $y(t_n)$. We discretize by theta method followed by a simple 3 point time filter. 
\begin{equation} \label{21}
\begin{split}
     \text{Step 1}: y_{n+1}^* &=y_n+k((1-\theta) f(t_n,y_n)+\theta f(t_{n+1},y_{n+1}^*)) \\
     \text{Step 2}: y_{n+1}&=y_{n+1}^*+\{a y_{n+1}^*+by_n+cy_{n-1}\}
     \end{split}
\end{equation}
 where $a, b, c \in \mathbb{R}$. Notice that for $\theta=0$, we get  explicit Forward Euler method, for $\theta=\frac{1}{2}$, we get Trapezoidal method (implicit) and for $\theta=1$ we get implicit Backward Euler method. We find proper values of $a,\ b, \ c$ for which the method is consistent in Section 3.1.
\subsection{Consistency and Accuracy}
First, we study consistency and accuracy.
\begin{Proposition} \label{p1c}
 Let the time step, $k$, be constant and $\nu\neq 2$. The method \eqref{21} is consistent if and only if $a=-\frac{\nu}{2},b=\nu,c=-\frac{\nu}{2}$ for some $\nu$. Thus, the step 2 of \eqref{21} is 
\begin{equation}\label{t2}
    y_{n+1}=y_{n+1}^*-\frac{\nu}{2}\{ y_{n+1}^*-2y_n+y_{n-1}\} 
\end{equation}
Moreover, when $\theta =  \frac{2+\nu}{4-2\nu}$ or equivalently $\nu = 2\frac{2\theta-1}{2\theta+1}$, \eqref{21} is second order convergent.
\end{Proposition}
\begin{proof}
 Rewriting the $Step\ 2$ we get $y_{n+1}^*=\frac{1}{1+a}(y_{n+1}-by_n-cy_{n-1})$. Putting this in $Step\ 1$, we get the following
\begin{align}\label{consistency}
   &\frac{1}{1+a}y_{n+1}-\frac{1+a+b}{1+a}y_n-\frac{c}{1+a}y_{n-1}
   =k(1-\theta)f(t_n,y_n)\notag\\&+k\theta f(t_{n+1},\frac{1}{1+a}y_{n+1}-\frac{b}{1+a}y_n-\frac{c}{1+a}y_{n-1})
\end{align}
Let $$y^*=\frac{1}{1+a}y(t_{n+1})-\frac{b}{1+a}y(t_n)-\frac{c}{1+a}y(t_{n-1})$$\\
Hence using Taylor Expansion, we get  
\begin{align*}
    &f(t_{n+1},y^*)=f(t_{n+1},y(t_{n+1}))+\frac{\partial f}{\partial y}(t_{n+1},y(t_{n+1}))(y^*-y(t_{n+1}))\\
     &+\frac{1}{2}\frac{\partial^2 f}{\partial y^2}(t_{n+1},y(t_{n+1}))(y^*-y(t_{n+1}))^2
     +\frac{1}{3!}\frac{\partial^3 f}{\partial y^3}(t_{n+1},y(t_{n+1}))(y^*-y(t_{n+1}))^3\\
     &+\cdots+\frac{1}{n!}\frac{\partial^n f}{\partial y^n}(t_{n+1},y(t_{n+1}))(y^*-y(t_{n+1}))^n+\cdots. 
\end{align*}
Notice, 
\begin{align*}
    y^*-y(t_{n+1})&=\frac{1}{1+a}y(t_{n+1})-\frac{b}{1+a}y(t_n)-\frac{c}{1+a}y(t_{n-1})-y(t_{n+1}),\\
    &=-\frac{a}{1+a}y(t_{n+1})-\frac{b}{1+a}y(t_n)-\frac{c}{1+a}y(t_{n-1}).
\end{align*}
Again by Taylor expansion, we get
$$y(t_n)=y(t_{n+1})-ky'(t_{n+1})+\frac{k^2}{2}y''(t_{n+1})+\frac{(-k)^3}{3!}y^{(3)}(t_{n+1})+\mathcal{O}(k^4).$$
$$y(t_{n-1})=y(t_{n+1})-2ky'(t_{n+1})+\frac{(2k)^2}{2}y''(t_{n+1})+\frac{(-2k)^3}{3!}y^{(3)}(t_{n+1})+\mathcal{O}(k^4).$$
Hence we get, 
\begin{align*}
    &y^*-y(t_{n+1})=-\frac{a+b+c}{1+a}y(t_{n+1})+\frac{b+2c}{1+a}ky'(t_{n+1})-\frac{b+4c}{2(1+a)}k^2y''(t_{n+1})\\
     &+\frac{1}{3!}\frac{b+8c}{1+a}k^3 y^{(3)}(t_{n+1})+\cdots+(-1)^{n+1}\frac{1}{n!}\frac{b+2^nc}{1+a}k^ny^{(n)}(t_{n+1})+\mathcal{O}(k^{n+1}).
\end{align*}
Insert the exact solution in \eqref{consistency}, we can get the local truncation error, 
\begin{align*}
&LTE=\frac{1}{1+a}y(t_{n+1})-\frac{1+a+b}{1+a}y(t_{n})-\frac{c}{1+a}y(t_{n-1})-k(1-\theta)y'(t_n)-k\theta f(t_{n+1},y^*) \\
&=\Big(\frac{1}{1+a}-\frac{1+a+b}{1+a}-\frac{c}{1+a}\Big)y(t_{n+1})+\Big(\frac{1+a+b}{1+a}+\frac{2c}{1+a}-1\Big)ky'(t_{n+1})+\mathcal{O}(k^2).
\end{align*}
 To prove the method is  consistent, we need to have
 \begin{equation*}
     \Big(\frac{1}{1+a}-\frac{1+a+b}{1+a}-\frac{c}{1+a}\Big)=0, \ \  \Big(\frac{1+a+b}{1+a}+\frac{2c}{1+a}-1\Big)=0
 \end{equation*}
 which implies two conditions $a+b+c=0, b+2c=0$ and $a\neq -1$. If we take $b=\nu$ as free variable, we get $a=-\frac{\nu}{2}, c=-\frac{\nu}{2}$.\\
 Thus we get the linear multistep method is consistency if and only if Step 2 reads
 \begin{equation*}
     y_{n+1}=y^*_{n+1}-\frac{\nu}2 \left( y^*_{n+1}-2y_n+y_{n-1}\right)
 \end{equation*}
 for some $\nu\in\mathbb{R}$. We need to investigate for higher order of convergence. Setting $b=\nu$, $a=-\frac{\nu}{2}, c=-\frac{\nu}{2}$ in \eqref{consistency} gives the equivalent method as
 \begin{equation}
 \begin{aligned}\label{ConstantTimeStepEquivalentMultistepMethod}
   &\frac{2}{2-\nu}y_{n+1}-\frac{2+\nu}{2-\nu}y_n-\frac{-\nu}{2-\nu}y_{n-1}=k(1-\theta)f(t_n,y_n)\\&+k\theta f(t_{n+1},\frac{2}{2-\nu}y_{n+1}-\frac{2\nu}{2-\nu}y_n-\frac{-\nu}{2-\nu}y_{n-1}).
   \end{aligned}
  \end{equation}
and 
$$y^*=\frac{2}{2-\nu}y(t_{n+1})-\frac{2\nu}{2-\nu}y(t_n)+\frac{\nu}{2-\nu} y(t_{n-1}).$$\\
Next,
 \begin{align*}
    y^*-y(t_{n+1})&=\frac{\nu}{2-\nu}k^2y''(t_{n+1})-\frac{\nu}{2-\nu}k^3 y^{(3)}(t_{n+1})+\cdots\\&+(-1)^{n+1}\frac{1}{n!}\frac{(2-2^{n})\nu}{2-\nu}k^ny^{(n)}(t_{n+1})+\mathcal{O}(k^{n+1}).
\end{align*}
and
 \begin{align*}
    f(t_{n+1},y^*)&=y'(t_{n+1})+y''(t_{n+1})(y^*-y(t_{n+1}))+\frac{1}{2}y'''(t_{n+1})(y^*-y(t_{n+1}))^2\\
     &+\frac{1}{3!}y^{(4)}(t_{n+1})(y^*-y(t_{n+1}))^3+\cdots+\frac{1}{n!}y^{(n+1)}(t_{n+1}))(y^*-y(t_{n+1}))^n+\cdots.
\end{align*}
and 
\begin{align*}
    f(t_n,y_n) &= y'(t_n) = y'(t_{n+1}) - k y''(t_{n+1}) + \mathcal{O}(k^2).
\end{align*}
Therefore
\begin{align*}
LTE&=\frac{2}{2-\nu}y(t_{n+1})-\frac{2+\nu}{2-\nu}y(t_{n})-\frac{-\nu}{2-\nu}y(t_{n-1})-k(1-\theta)y'(t_n)-k\theta f(t_{n+1},y^*), \\
& = -\frac{2+\nu}{2-\nu}\frac{k^2}{2}y''(t_{n+1})+\frac{\nu}{2-\nu}2k^2y''(t_{n+1})-k(1-\theta)(-k)y''(t_{n+1})+\mathcal{O}(k^3),\\
& = \frac{1}{2}(-\frac{2+\nu}{2-\nu}+\frac{4\nu}{2-\nu}+2(1-\theta))k^2y''(t_{n+1)}+\mathcal{O}(k^3),\\
& = \frac{1}{2}(\frac{3\nu-2}{2-\nu}+2(1-\theta))k^2y''(t_{n+1})+\mathcal{O}(k^3).
\end{align*}
 To get convergent of order 2, we need to have $\Big(\frac{3\nu-2}{2-\nu}+2(1-\theta)\Big)=0$ which implies 
 $\theta =  \frac{2+\nu}{4-2\nu}$ or equivalently $\nu = \frac{2(2\theta-1)}{2\theta+1}$.
\end{proof}

From the above proof, we see that for any choice of $\theta\in[0,1]$ and $\nu \neq 2$, we gain a consistent two-parameter family of methods \eqref{ConstantTimeStepEquivalentMultistepMethod}. Under the condition $2\theta \nu +\nu -4\theta +2=0$, the method becomes second order, but as it turns out their no choice of $\theta$ and $\nu$, where the method becomes third order, which can be seen in the results of Proposition \eqref{VariableTimeStepConsistency}.
\subsection{Constant time step: 0-stable and $A$-stability} \label{ConstantTimeStepStability}
The equivalent Linear multistep method \eqref{ConstantTimeStepEquivalentMultistepMethod} corresponds to a linear multistep method
\begin{equation}\label{24}
   \alpha_2 y_{n+1} +\alpha_1 y_{n}+\alpha_0 y_{n-1}=k(1-\theta) f(t_n,y_n)+k\theta f(t_{n+1},\beta_2 y_{n+1}+\beta_1 y_{n}+\beta_0 y_{n-1}) 
\end{equation}
where the coefficients are
\begin{align*}
    \alpha_2 &=\frac{1}{1-\frac{\nu}{2}},& \alpha_1 &=-\frac{1+\frac{\nu}{2}}{1-\frac{\nu}{2}},& \alpha_0 &=\frac{\frac{\nu}{2}}{1-\frac{\nu}{2}},\\
    \beta_2 &=\frac{1}{1-\frac{\nu}{2}}, &\beta_1 &=-\frac{\nu}{1-\frac{\nu}{2}},& \beta_0 &=\frac{\frac{\nu}{2}}{1-\frac{\nu}{2}}.
\end{align*}
and $ \nu \neq 2$.    
\begin{Proposition}
The method \eqref{21} is 0-stable for $-2 \leq \nu < 2$ and A-stable for $\theta\geq\frac{1}{2}$ and $ 2-4\theta\leq (2\theta+1)\nu\leq 4\theta-2$.
\end{Proposition}
\begin{proof}
    Consider the test function $y'=\lambda y$.
    Recall Equation $\eqref{24}$, we can get 
\begin{equation}\label{twostep}
   \alpha_2 y_{n+1}+\alpha_1 y_n+\alpha_0 y_{n-1}=k(1-\theta)\lambda y_n+k\theta \lambda(\beta_2 y_{n+1}+\beta_1 y_n+\beta_0 y_{n-1}).
\end{equation}
We can get the characteristic polynomials
\begin{align}
    \rho(\eta)&=\alpha_2 \eta^2+\alpha_1 \eta+\alpha_0, \label{rho}\\
    \sigma(\eta)&=(1-\theta)\eta+\theta(\beta_2 \eta^2+\beta_1 \eta+\beta_0) \notag \\
            &=\theta\beta_2\eta^2+(1-\theta+\theta\beta_1)\eta+\theta\beta_0. \label{sigma}
\end{align}
  The linear multistep method is $0$-stable if and only if all roots $z_i$ of the associated polynomial 
  $\rho(\eta)$, satisfy $|z_i|\leq 1$. It gives two roots $z_1=1,z_2=\frac{\nu}{2}$, hence for $0$-stability, we require $-1\leq \frac{ \nu}{2} <1$ which implies $-2\leq \nu <2$.\\
The linear multistep method is absolutely stable (A-stable) if for $Re(z) \leq 0$ the roots of
  \begin{equation*}
      \rho(\eta) - z \sigma(\eta)=0
  \end{equation*}
  satisfy $|\eta|\leq 1$.
We apply the Lemma \eqref{A-stableLemma}. The two step method \eqref{twostep} is A-stable if
  \begin{align}
      -\alpha_1&=\frac{1+\frac{\nu}{2}}{1-\frac{\nu}{2}}\geq 0, \label{1stCondition}\\
     1-2(1-\theta+\theta\beta_1)&= 1-2(1-\theta-\theta\frac{\nu}{1-\frac{\nu}{2}})\geq 0,\label{2ndCondition} \\
     2(\theta\beta_2-\theta\beta_0)+\alpha_1&=2(\theta\frac{1}{1-\frac{\nu}{2}}-\theta\frac{\frac{\nu}{2}}{1-\frac{\nu}{2}})-\frac{1+\frac{\nu}{2}}{1-\frac{\nu}{2}}\geq 0. \label{3rdCondition}
  \end{align}
  The first condition, \eqref{1stCondition}, holds if and only if $-2\leq \nu<2$; When $-2\leq \nu<2$, we can get the following:
  The second condition, \eqref{2ndCondition}, holds if and only if 
  $(1+2\theta)\nu\geq 2-4\theta$;
  The third condition, \eqref{3rdCondition}, holds if and only if 
  $(2\theta+1)\nu\leq 4\theta-2$.\\
  Hence we get when $\theta\geq\frac{1}{2}$ and $ 2-4\theta\leq (2\theta+1)\nu\leq 4\theta-2$, the two-step method \eqref{twostep} is A-stable. 
\end{proof}
\subsection{Constant time step: $A_0$ stability}
The following calculation is done to find more $A$-stability properties of the method.
Considering the roots of $\rho(\zeta)-k\lambda \sigma(\zeta)=0$ and setting $\zeta = e^{i\phi}$, we calculate $\frac{\rho(\zeta)}{\sigma(\zeta)}$.  
\begin{Proposition}
For $\zeta = e^{i\phi}$ the characteristic polynomials \eqref{rho} and \eqref{sigma} have a ratio of 
\begin{equation}
    \frac{\rho(\zeta)}{\sigma(\zeta)} = \frac{ (4(2\theta-1)+\nu^2(2\theta+1) - 8\nu\theta cos(\phi))(1-cos(\phi)) + (2-\nu)^2sin(\phi)i}{[2\theta cos(2\phi) + A cos(\phi) +\nu\theta]^2+[2\theta sin(2\phi)+A sin(\phi)]^2}. \label{CharacteristicRatio}
\end{equation}
\end{Proposition}
\begin{proof}
First multiply by numerator and denominator by $1-\nu/2$ to obtain
\begin{equation*}
    \frac{\rho(\zeta)}{\sigma(\zeta)} = \frac{\zeta^2-(1+\nu/2)\zeta+\nu/2}{\theta \zeta^2+[(1-\theta)(1-\nu/2)-\theta \nu]\zeta + \theta *\nu/2}.
\end{equation*}
To clear fractions multiple by 2,
\begin{equation*}
    \frac{\rho(\zeta)}{\sigma(\zeta)}=\frac{2\zeta^2-(2+\nu)\zeta + \nu}{2\theta \zeta^2 +[2-\nu-\theta(2+\nu)]\zeta +\theta \nu }.
\end{equation*}
Set $A=2-\nu-\theta(2+\nu)$. Next, substitute $\zeta = e^{i\phi}$.
\begin{equation*}
    \frac{\rho(\zeta)}{\sigma(\zeta)} = \frac{2[cos(2\phi)+i sin(2\phi)]-(2+\nu)[cos(\phi)+i sin(\phi)]+\nu}{2\theta[cos(2\phi)+i sin(2\phi)]+A[cos(\phi)+i sin(\phi]+\nu \theta}.
\end{equation*}
Grouping the real terms and the imaginary terms find
\begin{equation*}
    \frac{\rho(\zeta)}{\sigma(\zeta)} = \frac{[2cos(2\phi)-(2+\nu)cos(\phi)+\nu]+i[2 sin(2\phi)-(2+\nu)sin(\phi)]}{[2\theta cos(2\phi) + A cos(\phi) +\nu\theta]+i[2\theta sin(2\phi)+A sin(\phi)]}.
\end{equation*}
Next, rational the denominator by multiplying by the its conjugate. Let $D=[2\theta cos(2\phi) + A cos(\phi) +\nu\theta]^2+[2\theta sin(2\phi)+A sin(\phi)]^2$. Note that $D>0$.
\begin{align*}
    &\frac{\rho(\zeta)}{\sigma(\zeta)} =\frac{1}{D}
     \bigg[(2cos(2\phi)-(2+\nu)cos(\phi)+\nu)(2\theta cos(2\phi) + A cos(\phi) +\nu\theta)+(2 sin(2\phi)-(2+\nu)sin(\phi))\\&(2\theta sin(2\phi)+A sin(\phi)) +i\bigg((2 sin(2\phi)-(2+\nu)sin(\phi))(2\theta cos(2\phi) + A cos(\phi) +\nu\theta)\\&-(2cos(2\phi)-(2+\nu)cos(\phi)+\nu)(2\theta sin(2\phi)+A sin(\phi))\bigg)\bigg].
\end{align*}
 With a surprising amount of cancellation, the imaginary part simplifies as follows
\begin{gather*}
    Im(\frac{\rho(\zeta)}{\sigma(\zeta)}) = \frac{1}D\bigg[[2 sin(2\phi)-(2+\nu)sin(\phi)][2\theta cos(2\phi) + A cos(\phi) +\nu\theta]\\
    -[2cos(2\phi)-(2+\nu)cos(\phi)+\nu][2\theta sin(2\phi)+A sin(\phi)]\bigg] .
\end{gather*}
First, make use of the trig identities $sin(2\phi)=2sin(\phi)cos(\phi)$ and $cos(2\phi)=2cos^2(\phi)-1$ to find
\begin{gather*}
    Im(\frac{\rho(\zeta)}{\sigma(\zeta)}) = \frac{1}D\bigg[[4 cos(\phi) sin(\phi)-(2+\nu)sin(\phi)][4\theta cos^2(\phi) + A cos(\phi) +\theta(\nu-2)]\\
    -[4cos^2(\phi)-(2+\nu)cos(\phi)+\nu-2][4\theta sin(\phi)cos(\phi)+A sin(\phi)]\bigg] .
\end{gather*}
Next, factor out a $sin(\phi)$
\begin{gather*}
    Im(\frac{\rho(\zeta)}{\sigma(\zeta)}) = \frac{sin(\phi)}D\bigg[[4 cos(\phi)-(2+\nu)][4\theta cos^2(\phi) + A cos(\phi) +\theta(\nu-2)]\\
    -[4cos^2(\phi)-(2+\nu)cos(\phi)+\nu-2][4\theta cos(\phi)+A ]\bigg] .
\end{gather*}
Multiply out the terms
\begin{align*}
    &Im(\frac{\rho(\zeta)}{\sigma(\zeta)}) = \frac{sin(\phi)}D\bigg[
    [16\theta cos^3(\phi) + 4A cos^2(\phi) +4\theta (\nu-2) cos(\phi)-4\theta (2+\nu) cos^2(\phi) \\&-(2+\nu)A cos(\phi) -(2+\nu)\theta (\nu-2)]
    -[16\theta cos^3(\phi) -4\theta (2+\nu) cos^2(\phi)\\& +4\theta(\nu-2) cos(\phi) +4A cos^2(\phi) -(2+\nu)A cos(\phi) +(\nu-2)A]\bigg] .
\end{align*}
Group like terms to find
\begin{gather*}
    Im(\frac{\rho(\zeta)}{\sigma(\zeta)}) = \frac{sin(\phi)}D\bigg[(16\theta-16\theta)cos^3(\phi)+(4A-4\theta(2+\nu)+4\theta(2+\nu)-4A)cos^2(\phi) \\
    +(4\theta(\nu-2)-(2+\nu)A-4\theta(\nu-2)+(2+\nu)A)cos(\phi) + (-\theta(\nu-2)(\nu+2)-(\nu-2)A) \bigg].
\end{gather*}
Simplifying 
\begin{gather*}
    Im(\frac{\rho(\zeta)}{\sigma(\zeta)}) = \frac{sin(\phi)}D[-\theta(\nu-2)(\nu+2)-(\nu-2)A].
\end{gather*}
Factor out $(\nu-2)$ and recall the definition of $A=2-\nu-2\theta-\theta\nu$ to find
\begin{gather*}
    Im(\frac{\rho(\zeta)}{\sigma(\zeta)}) = \frac{(\nu-2)sin(\phi)}D[-\theta(\nu+2)-(2-\nu-2\theta-\theta\nu)].
\end{gather*}
Simplify the remaining terms yields
\begin{gather*}
    Im(\frac{\rho(\zeta)}{\sigma(\zeta)}) = \frac{(\nu-2)sin(\phi)}D[-2+\nu].
\end{gather*}
Thus
\begin{gather*}
    Im(\frac{\rho(\zeta)}{\sigma(\zeta)}) = \frac{(2-\nu)^2sin(\phi)}D.
\end{gather*}
We show the details for the real part.
\begin{align*}
Re( \frac{\rho(\zeta)}{\sigma(\zeta)} )=& \frac{1}{D}\bigg[[2cos(2\phi)-(2+\nu)cos(\phi)+\nu][2\theta cos(2\phi) + A cos(\phi) +\nu\theta]\\&+[2 sin(2\phi)-(2+\nu)sin(\phi)][2\theta sin(2\phi)+A sin(\phi)]\bigg],
\\=&\frac{1}D[4\theta (cos^2(2\phi)+sin^2(2\phi)) - (2+\nu)A(cos^2(\phi)+sin^2(\phi)) \\&+ (2A-2\theta (2+\nu))(cos(2\phi)cos(\phi)+sin(2\phi)sin(\phi))\\& + 4\nu\theta cos(2\phi) + (\nu A-\nu\theta (2+\nu)) cos(\phi) +\theta \nu^2].
\end{align*}
Recall the Pythagorean identity and the double angle identities $cos(2\phi)cos(\phi)+sin(2\phi)sin(\phi)=cos(\phi)$ and $cos(2\phi)=2cos^2(\phi)-1$. Applying them yields
\begin{gather*}
   Re( \frac{\rho(\zeta)}{\sigma(\zeta)} )= \frac{1}D[(4\theta-(2+\nu)A+\theta\nu^2-4\nu\theta) + (2A-2\theta(2+\nu)\\+\nu A - \nu \theta (2+\nu))cos(\phi)+8\nu\theta cos^2(\phi)].
\end{gather*}
Notice,  $$4\theta-(2+\nu)A+\theta\nu^2-4\nu\theta = 4(2\theta-1)+\nu^2(2\theta+1),$$ and the cosine coefficient simplifies to $$2A-2\theta(2+\nu)+\nu A - \nu \theta (2+\nu) = 4(1-2\theta)-\nu^2 (2\theta+1)-8\theta \nu.$$ Therefore 
\begin{equation*}
    Re( \frac{\rho(\zeta)}{\sigma(\zeta)} )= \frac{1}D[4(2\theta-1)+\nu^2(2\theta+1)-[4(2\theta-1)+\nu^2(2\theta+1)+8\nu\theta]cos(\phi) + 8\nu\theta cos^2(\phi)].
\end{equation*}
Observe that $\phi=0$ implies $\zeta=1$ and recall $\rho(1)=0$. Therefore $\phi=0$ will make the expression 0, which indicates that the expression will factor with $1-cos(\phi)$ as one of the factors. We move forward with factor by grouping
\begin{equation*}
    Re( \frac{\rho(\zeta)}{\sigma(\zeta)} )= \frac{1}D[ (4(2\theta-1)+\nu^2(2\theta+1))(1-cos(\phi)) + 8\nu\theta cos(\phi)(cos(\phi)-1)].
\end{equation*}
Factoring out the $1-cos(\phi)$
\begin{equation*}
    Re( \frac{\rho(\zeta)}{\sigma(\zeta)} )= \frac{1}D[ (4(2\theta-1)+\nu^2(2\theta+1) - 8\nu\theta cos(\phi))(1-cos(\phi))].
\end{equation*}
\end{proof}\\
Note $1-cos(\phi)\geq 0$, $\forall \phi$. Hence $Re( \frac{\rho(\zeta)}{\sigma(\zeta)} )\geq 0$, $\forall \phi$ is equivalent to $$4(2\theta-1)+\nu^2(2\theta+1)-8\nu\theta cos(\phi)\geq 0, \hspace{.5cm}\forall \phi$$ which is equivalent to $$4(2\theta-1)+\nu^2(2\theta+1)-8 |\nu| \theta\geq 0.$$ Recall at the beginning of the analysis on the ratio of characteristic polynomials we multiplied by $1-\nu/2$ and then by $2$. Notice that $\nu = 2$ makes the expression $0$. We factor out $2-|\nu|$.
\begin{equation*}
    [2(2\theta-1)-|\nu|(2\theta+1)](2-|\nu|) \geq 0.
\end{equation*}
Since $-2\leq \nu \leq 2$ the requirement becomes
\begin{equation*}
    2(2\theta-1)-|\nu|(2\theta+1)\geq 0,
\end{equation*}
which is
\begin{equation*} 
    |\nu| \leq 2 \cdot \frac{2\theta-1}{2\theta+1},
\end{equation*}
or equivalently
\begin{equation*}
    \theta \geq \frac{1}2 \cdot \frac{2+|\nu|}{2-|\nu|}.
\end{equation*}
This condition implies $\theta \geq \frac{1}2$ is necessary for $A$-stability, which is consistent with Dahlquist \cite{dahlquist1979some}. When the condition is written as 
$$-2 \cdot \frac{2\theta-1}{2\theta+1}\leq \nu \leq 2 \cdot \frac{2\theta-1}{2\theta+1} $$ one accomplishes 2nd order if the right-hand side has equality. Also, when the left-hand side is satisfied, we obtain $A_0$-stability, which is shown below.
\begin{Proposition}
The $\theta$-method with time filter (\ref{21}) is $A_0$-stable if $\nu > -2 \cdot \frac{2\theta-1}{2\theta+1}$. 
\end{Proposition}
\begin{proof}
The imaginary part of $\frac{\rho(\zeta)}{\sigma(\zeta)}$ is $(2-\nu)^2 sin(\phi)$, which implies the boundary of the stability region crosses the real axis when $\phi =0$ and $\phi =\pi$. Since $\phi =0$ corresponds to the known root $\zeta =1$, we focus on the intersection at $\phi = \pi$. The result is obtained by evaluating the real part of $\frac{\rho(\zeta)}{\sigma(\zeta)}$
\begin{gather*}
    D|_{\phi = \pi}=[2\theta cos(2\pi)+A cos(\pi) +\nu\theta]^2+[2\theta sin(2\pi) +A sin(\pi)]^2 = [2\theta -A + \nu \theta]^2.
\end{gather*}
Substituting $D$ into the simplified real part we find
\begin{equation*}
    Re(\frac{\rho(\zeta)}{\sigma(\zeta)})|_{\phi=\pi} = \frac{4(2\theta-1)+\nu^2(2\theta+1)-8\nu\theta cos(\pi)](1-cos(\pi)}{[2\theta -A + \nu \theta]^2}.
\end{equation*}
By factoring the numerator we find 
\begin{equation*}
    Re(\frac{\rho(\zeta)}{\sigma(\zeta)})|_{\phi=\pi} = \frac{2(2+\nu)[(2\theta+1)\nu+2(2\theta-1)]}{[2\theta -A + \nu \theta]^2}.
\end{equation*}
Recalling the definition of $A$ we find
\begin{equation*}
    Re(\frac{\rho(\zeta)}{\sigma(\zeta)})|_{\phi=\pi} = \frac{2(2+\nu)[(2\theta+1)\nu+2(2\theta-1)]}{[(2\theta+1)\nu+2(2\theta-1)]^2}.
\end{equation*}
Thus
\begin{equation*}
    Re(\frac{\rho(\zeta)}{\sigma(\zeta)})|_{\phi=\pi} = \frac{2(2+\nu)}{(2\theta+1)\nu+2(2\theta-1)}.
\end{equation*}
Since $-2 \leq \nu \leq 2$ the condition $ Re(\frac{\rho(\zeta)}{\sigma(\zeta)})|_{\phi=\pi} \geq 0$ becomes $(2\theta+1)\nu+2(2\theta-1) > 0$ which is equivalent to $ \nu > -2\cdot \frac{2\theta-1}{2\theta+1}$.
\end{proof}
From this result observe that the forward Euler method cannot be made $A_0$-stable for any feasible values of $\nu$,
\begin{equation*}
    Re(\frac{\rho(\zeta)}{\sigma(\zeta)})|_{\phi=\pi, \theta=0} = \frac{2(2+\nu)}{\nu-2}.
\end{equation*}
One can increase the amount of the negative real axis inside the stability region by choosing $\nu$ near 2 in this case, but the local truncation error is multiplied by a factor $\frac{1}{2-\nu}$. Hence choosing $\nu$ near 2 may result in a devastating amount of error. 

\subsection{Stability Regions}
\begin{figure} [H]
\center
\includegraphics[width=0.75\textwidth]{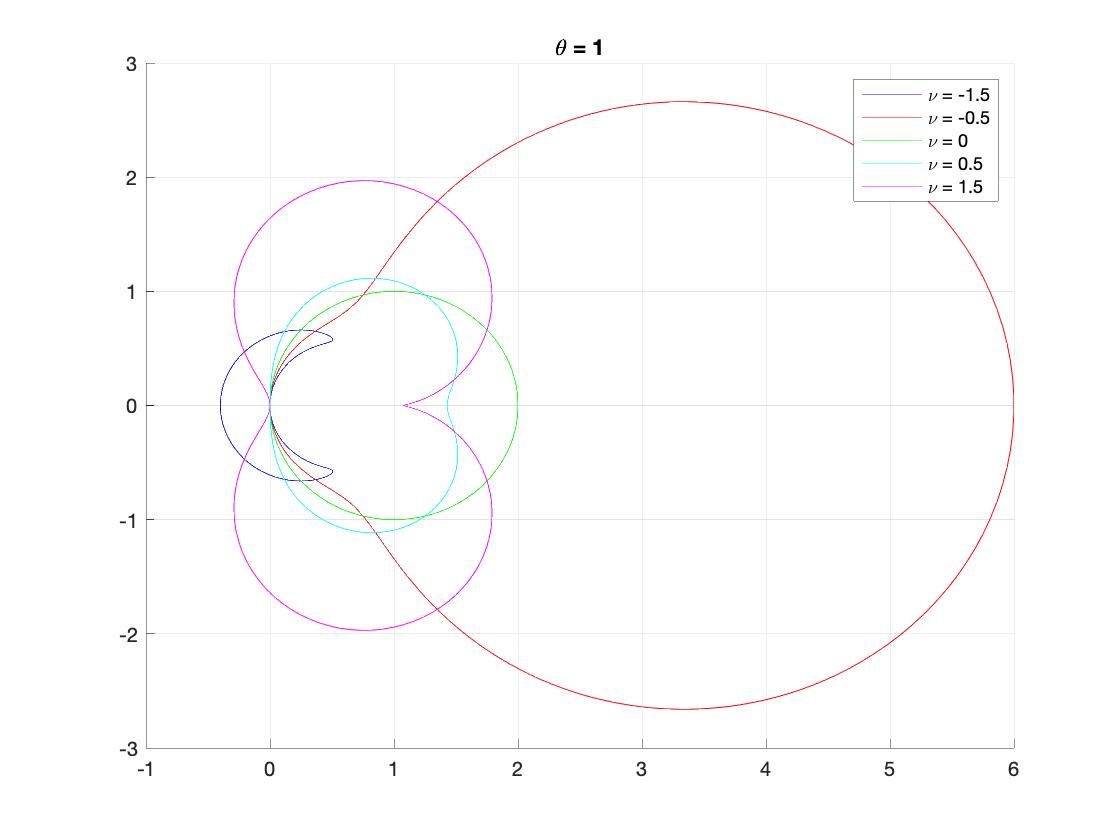}
\caption{Stability Regions for Backward Euler plus filter for several values of $\nu$. The region is outside the curves for $\nu > -\frac{2}3$ and inside for $\nu < -\frac{2}3$.}
\label{SRegionBE+Filter}
\end{figure}

\begin{figure} [H]
\includegraphics[width=0.75\textwidth]{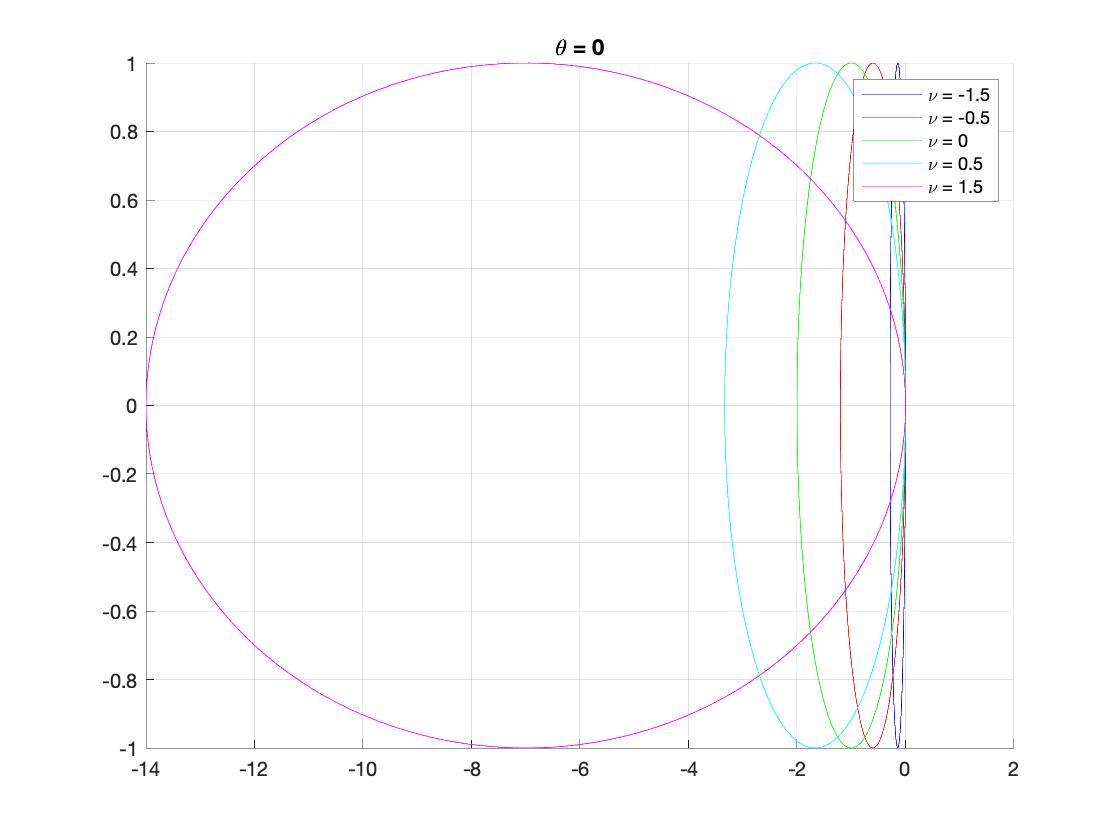}
\caption{Stability Regions for Forward Euler plus filter for several values of $\nu$. The region is inside the curves.}
\label{SRegionFE+Filter}
\end{figure}

\begin{figure} [H]
\includegraphics[width=0.75\textwidth]{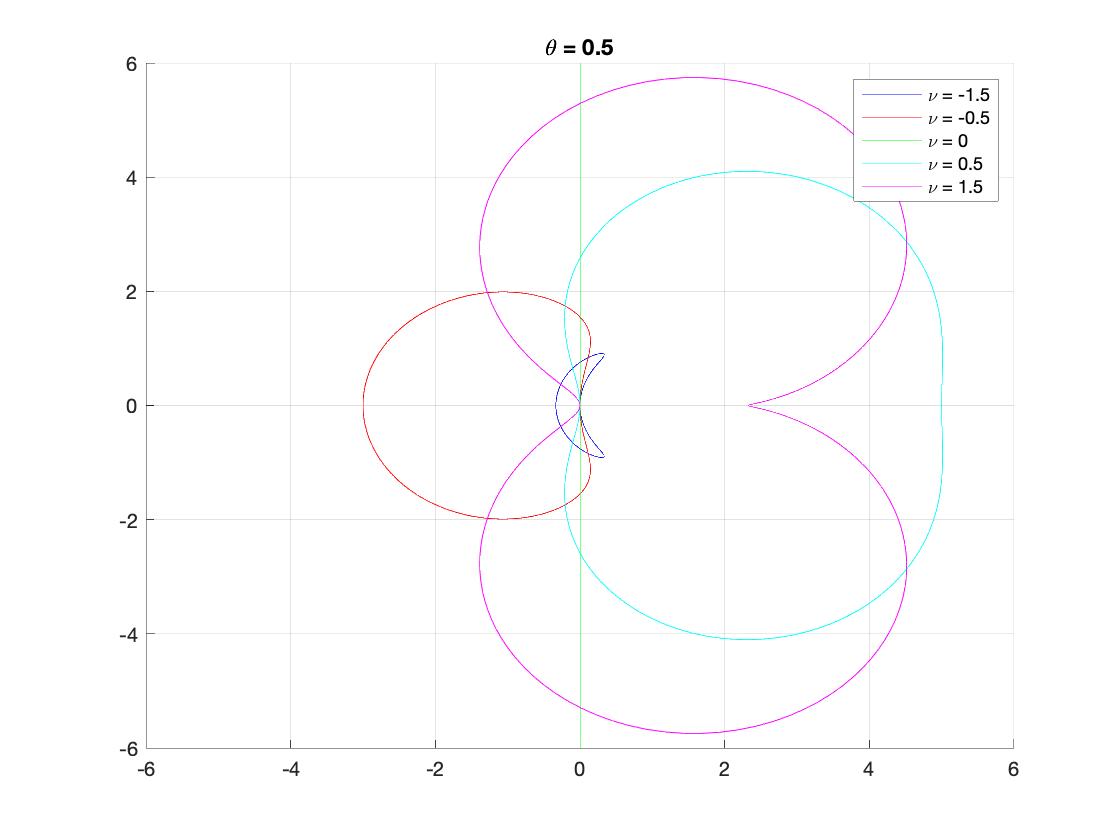}
\caption{Stability Regions for trapezoid method plus filter for several values of $\nu$. The region is inside the curves with $\nu < 0$ and outside the curves with $\nu >0$. The region for $\nu =0$ is the left half-plane.}
\label{SRegionTrap+Filter}
\end{figure}
All stability regions are consistent with the theory. 
Figure \eqref{SRegionBE+Filter} shows the method is A-stable for Backward Euler plus filter for $-\frac{2}3 < \nu < \frac{2}3$. Recall $\nu=\frac{2}3$ is 2nd order and A-stable. Note $ \frac{2}3 <\nu < 2$ is $A_0$-stable and not $A_0$-stable for $-2<\nu<\frac{-2}3$.
Figure \eqref{SRegionFE+Filter} shows the method is not A-stable or $A_0$-stable for Forward Euler plus filter for $-2 < \nu <2$. Notice the stability region grows as $\nu$ increases in size.
Figure \eqref{SRegionTrap+Filter} shows the method is A-stable for Trapezoid Rule plus filter if and only if $\nu = 0$. If $\nu >0$, then the method is $A_0$-stable. If $\nu<0$ the method is not $A_0$-stable, and the stability region shrinks as $\nu$ decreases.

\section{Variable time step}
In this section, we consider variable time step $k_n$.
 We consider $\theta$-method plus a general 3-point time filter described here:
\begin{equation} \label{41}
\begin{split}
     \text{Step 1}: y_{n+1}^* &=y_n+k_n((1-\theta) f(t_n,y_n)+k_n \theta f(t_{n+1},y_{n+1}^*)) \\
     \text{Step 2}: y_{n+1}&=y_{n+1}^*+\{a y_{n+1}^*+by_n+cy_{n-1}\}
     \end{split}
\end{equation}
\subsection{Consistency and Accuracy}
First, we study consistency and accuracy.
\begin{Proposition}\label{VariableTimeStepConsistency}
Consider the variable time step with $\tau=\frac{k_n}{k_{n-1}}$ and $\nu\neq 1+\tau$. The method \eqref{41} is consistent if and only if  $a=-\frac{\nu}{1+\tau},b=\nu, c=-\frac{\tau\nu}{1+\tau}$ for some $\nu$. Thus, the step 2 of \eqref{41} is 
\begin{equation*}
    y_{n+1}=y^*_{n+1}-\frac{\nu}{1+\tau} \left( y^*_{n+1}- (1+\tau) y_n+\tau y_{n-1}\right)
\end{equation*}
Moreover, when $ \theta = \frac{\nu + \tau + \tau^2}{2\tau (1-\nu+\tau)}$or equivalently $ \nu = \frac{\tau(1+\tau)(2\theta-1)}{2\theta\tau+1}$, \eqref{41} is second order convergent and the local truncation error (LTE) is
\begin{align*}
LTE&=\big(\frac{1+\tau+\nu\tau}{1-\nu+\tau}\frac{\tau^3}{6}-\frac{\nu\tau}{1-\nu+\tau}\frac{(1+\tau)^3}{6}+\frac{\tau-2\nu\tau+\tau^2-\nu}{(1-\nu+\tau)\tau}\frac{\tau^3}{2}\big)k_{n-1}^3y'''(t_{n+1})\\
&-(\nu+\nu\tau)\frac{\nu\tau(1+\tau)^2-\nu\tau^2(1+\tau)}{2(1-\nu+\tau)^2}k_{n-1}^3(y''(t_{n+1}))^2 +\mathcal{O}(k_{n-1}^4).
\end{align*}
\end{Proposition}

\begin{proof}
 Rewriting the $Step\ 2$ we get $y_{n+1}^*=\frac{1}{1+a}(y_{n+1}-by_n-cy_{n-1})$. Putting this in $Step\ 1$, we get the following
\begin{equation}\label{variable}
\begin{aligned}
   &\frac{1}{1+a}y_{n+1}-\frac{1+a+b}{1+a}y_n-\frac{c}{1+a}y_{n-1}\\&=k_n(1-\theta)f(t_n,y_n)+k_n\theta f(t_{n+1},\frac{1}{1+a}y_{n+1}-\frac{b}{1+a}y_n-\frac{c}{1+a}y_{n-1}).
   \end{aligned}
\end{equation}
Let $$y^*=\frac{1}{1+a}y(t_{n+1})-\frac{b}{1+a}y(t_n)-\frac{c}{1+a}y(t_{n-1}).$$
Hence using Taylor Expansion, we get  
\begin{align*}
    f(t_{n+1},y^*)&=f(t_{n+1},y(t_{n+1}))+\frac{\partial f}{\partial y}(t_{n+1},y(t_{n+1}))(y^*-y(t_{n+1}))\\
     &+\frac{1}{2}\frac{\partial^2 f}{\partial y^2}(t_{n+1},y(t_{n+1}))(y^*-y(t_{n+1}))^2
     +\frac{1}{3!}\frac{\partial^3 f}{\partial y^3}(t_{n+1},y(t_{n+1}))(y^*-y(t_{n+1}))^3\\
     &+\cdots+\frac{1}{n!}\frac{\partial^n f}{\partial y^n}(t_{n+1},y(t_{n+1}))(y^*-y(t_{n+1}))^n+\cdots.
\end{align*}
Notice 
\begin{align*}
    y^*-y(t_{n+1})&=\frac{1}{1+a}y(t_{n+1})-\frac{b}{1+a}y(t_n)-\frac{c}{1+a}y(t_{n-1})-y(t_{n+1}),\\
    &=-\frac{a}{1+a}y(t_{n+1})-\frac{b}{1+a}y(t_n)-\frac{c}{1+a}y(t_{n-1}).
\end{align*}
Let $\tau=\frac{k_n}{k_{n-1}}$. By doing Taylor expansion, we get
\begin{equation*}
    \begin{array}{ll}
         y(t_n)&=y(t_{n+1})-k_n y'(t_{n+1})+\frac{k_n^2}{2}y''(t_{n+1})+\frac{(-k_n)^3}{3!}y^{(3)}(t_{n+1})+\mathcal{O}(k_{n-1}^4),\\
         &=y(t_{n+1})-k_{n-1}\tau y'(t_{n+1})+\frac{k_{n-1}^2\tau^2}{2}y''(t_{n+1})+\frac{(-k_{n-1}\tau)^3}{3!}y^{(3)}(t_{n+1})+\mathcal{O}(k_{n-1}^4).
    \end{array}
\end{equation*}
and
\begin{equation*}
    \begin{array}{ll}
         y(t_{n-1})&=y(t_{n+1})-(k_n+k_{n-1}) y'(t_{n+1})+\frac{(k_n+k_{n-1})^2}{2}y''(t_{n+1})\\&+\frac{(-k_n-k_{n-1})^3}{3!}y^{(3)}(t_{n+1})+\mathcal{O}(k_{n-1}^4),  \\
         &= y(t_{n+1})-k_{n-1}(1+\tau) y'(t_{n+1})+\frac{(k_{n-1}(1+\tau))^2}{2}y''(t_{n+1})\\&+\frac{-(k_{n-1}(1+\tau))^3}{3!}y^{(3)}(t_{n+1})+\mathcal{O}(k_{n-1}^4).
    \end{array}
\end{equation*}
and
\begin{equation*}
\begin{array}{ll}
         f(t_n,y(t_n))=y'(t_n)&=y'(t_{n+1})-k_n y''(t_{n+1})+\frac{k_n^2}{2}y^{(3)}(t_{n+1})-\frac{k_n^3}{3!}y^{(4)}(t_{n+1})+\mathcal{O}(k_{n-1}^4),  \\
         &=y'(t_{n+1})-k_{n-1}\tau y''(t_{n+1})+\frac{k_{n-1}^2\tau^2}{2}y^{(3)}(t_{n+1})\\&-\frac{k_{n-1}^3\tau^3}{3!}y^{(4)}(t_{n+1})+\mathcal{O}(k_{n-1}^4).
\end{array}
\end{equation*}
Hence we get 
\begin{align*}
    y^*-y(t_{n+1})&=-\frac{a+b+c}{1+a}y(t_{n+1})+\frac{b\tau+c(1+\tau)}{1+a}k_{n-1}y'(t_{n+1})-\frac{b\tau^2+c(1+\tau)^2}{2(1+a)}k_{n-1}^2y''(t_{n+1})\\
     &+\frac{1}{3!}\frac{b\tau^3+c(1+\tau)^3}{(1+a)}k_{n-1}^3 y^{(3)}(t_{n+1})+\cdots\\
     &+(-1)^{N+1}\frac{1}{N!}\frac{b\tau^N+c(1+\tau)^N}{(1+a)}k_{n-1}^Ny^{(N)}(t_{n+1})+\mathcal{O}(k_{n-1}^{N+1}).
\end{align*}
Insert the exact solution in \eqref{variable} to get the local truncation error, 

\begin{align*}
LTE&=\frac{1}{1+a}y(t_{n+1})-\frac{1+a+b}{1+a}y(t_{n})-\frac{c}{1+a}y(t_{n-1})-k_n(1-\theta)y'(t_n)-k_n\theta f(t_{n+1},y^*), \\
    &=\Big(\frac{1}{1+a}-\frac{1+a+b}{1+a}-\frac{c}{1+a}\Big)y(t_{n+1})+\Big(\frac{1+a+b}{1+a}\tau+\frac{c}{1+a}(1+\tau)-\tau\Big)k_{n-1}y'(t_{n+1})\\&+\mathcal{O}(k_{n-1}^2).
\end{align*}
 To prove the method is consistent, we need to have
 \begin{equation*}
     \Big(\frac{1}{1+a}-\frac{1+a+b}{1+a}-\frac{c}{1+a}\Big)=0, \Big(\frac{1+a+b}{1+a}\tau+\frac{c}{1+a}(1+\tau)-\tau\Big)=0 
 \end{equation*} which implies two conditions $a+b+c=0, b\tau+c+c\tau=0$ and $a\neq -1$. If we take $b=\nu$ as free variable, we get $a=-\frac{\nu}{1+\tau}, c=-\frac{\tau\nu}{1+\tau}$.\\
 Thus we get the consistent equivalent linear multistep method as
 \begin{gather*}
     \frac{1+\tau}{1+\tau-\nu} y_{n+1} - \frac{1+\tau+\nu \tau}{1+\tau -\nu} y_n +\frac{\tau \nu}{1+\tau-\nu} y_{n-1} \\= k_n (1-\theta) f(t_n,y_n)  + k_n \theta f(t_{n+1}, \frac{1+\tau}{1+\tau-\nu} y_{n+1} - \frac{\nu+\nu \tau}{1+\tau -\nu} y_n +\frac{\tau \nu}{1+\tau-\nu} y_{n-1}).
 \end{gather*}
 We need to investigate for higher order convergence. We already have $b=\nu$, $a=-\frac{\nu}{1+\tau},\ c=-\frac{\tau\nu}{1+\tau}$ for consistency and therefore
 \begin{align*}
    y^*-y(t_{n+1})&=\frac{\nu\tau(1+\tau)^2-\nu\tau^2(1+\tau)}{2(1-\nu+\tau)}k_{n-1}^2y''(t_{n+1})-\frac{1}{3!}\frac{\nu\tau^3(1+\tau)-\nu\tau(1+\tau)^2}{2(1-\nu+\tau)}k_{n-1}^3 y^{(3)}(t_{n+1})\\&+\cdots+(-1)^{N+1}\frac{1}{N!}\frac{\nu\tau^N(1+\tau)-\nu\tau(1+\tau)^N}{1-\nu+\tau}k_{n-1}^N y^{(N)}(t_{n+1})+\mathcal{O}(k_{n-1}^{N+1}).
\end{align*}
and
 \begin{align*}
    f(t_{n+1},y^*)&=y'(t_{n+1})+y''(t_{n+1})(y^*-y(t_{n+1}))+\frac{1}{2}y'''(t_{n+1})(y^*-y(t_{n+1}))^2\\
     &+\frac{1}{3!}y^{(4)}(t_{n+1})(y^*-y(t_{n+1}))^3+\cdots+\frac{1}{n!}y^{(n+1)}(t_{n+1}))(y^*-y(t_{n+1}))^n+\cdots.
\end{align*}
The local truncation errors simplifies to
\begin{align*}
LTE&=\frac{1+\tau}{1-\nu+\tau}y(t_{n+1})-\frac{1+\tau+\nu\tau}{1-\nu+\tau}y(t_{n})+\frac{\nu\tau}{1-\nu+\tau}y(t_{n-1})\\&-k_{n-1}\tau(1-\theta)y'(t_n)-k_{n-1}\tau\theta f(t_{n+1},y^*),\\
&=\big(-\frac{1+\tau+\nu\tau}{1-\nu+\tau}\frac{\tau^2}{2}+\frac{\nu\tau}{1-\nu+\tau}\frac{(1+\tau)^2}{2}+(1-\theta)\tau^2\big)k_{n-1}^2y''(t_{n+1})+\mathcal{O}(k_{n-1}^3),\\
&=\frac{1}{2}\big(\frac{(\nu+2\nu\tau-\tau-\tau^2)\tau}{1-\nu+\tau}+2(1-\theta)\tau^2\big)k_{n-1}^2y''(t_{n+1})+\mathcal{O}(k_{n-1}^3).
\end{align*}
 To get second order convergence, we need to have $\Big(\frac{(\nu+2\nu\tau-\tau-\tau^2)\tau}{1-\nu+\tau}+2(1-\theta)\tau^2\Big)=0$ which implies $ \theta = \frac{\nu + \tau + \tau^2}{2\tau (1-\nu+\tau)}$ or equivalently $ \nu = \frac{\tau(1+\tau)(2\theta-1)}{2\theta\tau+1}$. In this case, we find
\begin{align*}
LTE&=\frac{1+\tau}{1-\nu+\tau}y(t_{n+1})-\frac{1+\tau+\nu\tau}{1-\nu+\tau}y(t_{n})+\frac{\nu\tau}{1-\nu+\tau}y(t_{n-1})\\&-k_{n-1}\tau(1-\theta)y'(t_n)-k_{n-1}\tau\theta f(t_{n+1},y^*),\\
&=\big(\frac{1+\tau+\nu\tau}{1-\nu+\tau}\frac{\tau^3}{6}-\frac{\nu\tau}{1-\nu+\tau}\frac{(1+\tau)^3}{6} - (1-\theta)\frac{\tau^3}{2}\big)k_{n-1}^3y'''(t_{n+1})\\
&-\theta \tau\frac{\nu\tau(1+\tau)^2-\nu\tau^2(1+\tau)}{2(1-\nu+\tau)}k_{n-1}^3(y''(t_{n+1}))^2 +\mathcal{O}(k_{n-1}^4),\\
&=\big(\frac{1+\tau+\nu\tau}{1-\nu+\tau}\frac{\tau^3}{6}-\frac{\nu\tau}{1-\nu+\tau}\frac{(1+\tau)^3}{6} -\frac{\tau-2\nu\tau+\tau^2-\nu}{2(1-\nu+\tau)\tau}\frac{\tau^3}{2}\big)k_{n-1}^3y'''(t_{n+1})\\
&-(\nu+\nu\tau)\frac{\nu\tau(1+\tau)^2-\nu\tau^2(1+\tau)}{2(1-\nu+\tau)^2}k_{n-1}^3(y''(t_{n+1}))^2 +\mathcal{O}(k_{n-1}^4).
\end{align*}
which can be further simplified to 
\begin{align*}
LTE &=-\frac{\tau (\nu (2+3\tau)+\tau^2(\tau+1))}{12(1+\tau-\nu)}k_{n-1}^3y'''(t_{n+1})\\
&-\frac{\nu\tau (\nu+\tau+\tau^2)(\tau+1)}{4(1+\tau-\nu)^2} k_{n-1}^3(y''(t_{n+1}))^2 +\mathcal{O}(k_{n-1}^4).\\
\end{align*}
From here, it is clear that there is no choice of $\nu$ that will make both $\mathcal{O}(k_{n-1}^3)$ terms equal to 0, the method cannot achieve 3rd order. 
 \end{proof}

 \subsection{Variable time step: Stability} \label{VariableTimeStepStability}
 To maintain the consistency of \eqref{41}, we consider the following
\begin{equation}\label{variabletime}
\begin{split}
      \text{Step 1}: y_{n+1}^* &=y_n+k_n((1-\theta) f(t_n,y_n)+\theta f(t_{n+1},y_{n+1}^*)) \\
     \text{Step 2}:y_{n+1}&=y^*_{n+1}+\frac{-\nu}{1+\tau}y^*_{n+1}+\nu y_n-\frac{\nu\tau}{1+\tau}y_{n-1}
     \end{split}
\end{equation}
We can derive a linear multistep method from \eqref{variabletime}
\begin{equation}\label{vlinear}
   \alpha_2 y_{n+1} +\alpha_1 y_{n}+\alpha_0 y_{n-1}=k_n(1-\theta) f(t_n,y_n)+k_n\theta f(t_{n+1},\beta_2 y_{n+1}+\beta_1 y_{n}+\beta_0 y_{n-1}) 
\end{equation}
This corresponds to the linear multistep method \eqref{vlinear} with coefficients
\begin{equation*}
    \begin{array}{ccc}
         \alpha_2=\frac{1+\tau}{1+\tau-\nu}, & \alpha_1=-\frac{1+\tau+\tau\nu}{1+\tau-\nu}, & \alpha_0=\frac{\tau\nu}{1+\tau-\nu},\\
         \beta_2=\frac{1+\tau}{1+\tau-\nu}, &\beta_1=-\frac{\nu+\tau\nu}{1+\tau-\nu}, &\beta_0=\frac{\tau\nu}{1+\tau-\nu}. 
    \end{array}
\end{equation*}

\begin{Proposition}
The method \eqref{variabletime} is 0-stable for
$-\frac{1+\tau}{\tau}\leq \nu < \frac{1+\tau}{\tau}$ and A-stable for  $\theta\geq \frac{1}{2}$, $\tau>0$ and  $$
- \frac{(2\theta-1)(1+\tau)}{1+2\theta\tau}\leq \nu\leq \min\{1+\tau,\frac{(2\theta-1)(1+\tau)}{(1+2\theta)\tau}\}
.$$
\end{Proposition}

\begin{proof}
    Consider the test function $y'=\lambda y$.
    Recall Equation $\eqref{vlinear}$, we can get 
\begin{equation}\label{twostepv}
   \alpha_2 y_{n+1}+\alpha_1 y_n+\alpha_0 y_{n-1}=k_{n}(1-\theta)\lambda y_n+k_{n}\theta \lambda(\beta_2 y_{n+1}+\beta_1 y_n+\beta_0 y_{n-1}).
\end{equation}
The linear multistep method is $0-stable$ if and only if all roots $z_i$ of the associated polynomial 
  $\rho(\eta)=0$,  satisfy $|z_i|\leq 1$.\\
  It gives two roots $z_1=1,z_2=\frac{\tau\nu}{1+\tau}$. Hence, for $0$-stability, we require $-1\leq \frac{\tau\nu}{1+\tau} <1$ which implies $-\frac{1+\tau}{\tau}\leq \nu < \frac{1+\tau}{\tau}$.\\
  To prove the linear multistep method is absolutely stable, we need $|y_{n+1}|\leq |y_n|$ for those values $z=k\lambda$. 
  This corresponds to values for which all values of \eqref{eta} satisfy $|\eta|\leq 1$.
  \begin{equation}\label{eta}
      \rho(\eta)=k\lambda\sigma(\eta).
  \end{equation}
  
  The $A$-stability of general two step method is characterized in terms of their coefficients. We apply the Lemma 1 from \cite{dahlquist1979some}. The two step method \eqref{twostepv} is A-stable if
  \begin{align*}
      -\alpha_1&=\frac{1+\tau+\tau\nu}{1+\tau-\nu}\geq 0,\\
     1-2(1-\theta+\theta\beta_1)&= 1-2(1-\theta-\theta\frac{\nu+\tau\nu}{1+\tau-\nu})\geq 0,\\
     2(\theta\beta_2-\theta\beta_0)+\alpha_1&=2(\theta\frac{1+\tau}{1+\tau-\nu}-\theta\frac{\tau\nu}{1+\tau-\nu})-\frac{1+\tau+\tau\nu}{1+\tau-\nu}\geq 0.
  \end{align*}
  The first condition holds if and only if 
  \begin{equation}
      -\frac{1+\tau}{\tau}\leq \nu<1+\tau. \label{VarStability1stCondition}
  \end{equation}
  The second condition holds if 
  \begin{equation*}
      \theta \frac{1+\tau+\tau\nu}{1+\tau-\nu}\geq \frac{1}{2} \notag
  \end{equation*}
  which holds if and only if 
  \begin{equation} \label{VarStability2ndCondition}
      \nu\geq \frac{1-2\theta+\tau-2\theta\tau}{1+2\theta\tau}=\frac{(1-2\theta)(1+\tau)}{1+2\theta\tau}.
  \end{equation}
  As $\nu<1+\tau$, the third condition hold if
  \begin{equation} \label{VarStability3rdCondition}
      \nu\leq \frac{-1+2\theta-\tau+2\theta\tau}{\tau+2\theta\tau}=\frac{(2\theta-1)(1+\tau)}{(1+2\theta)\tau}.
  \end{equation}
 From \eqref{VarStability2ndCondition} and \eqref{VarStability3rdCondition}, we need the following result
 \begin{equation} \label{2nd3rd}
      \frac{(1-2\theta)(1+\tau)}{1+2\theta\tau} \leq \frac{(2\theta-1)(1+\tau)}{(1+2\theta)\tau}.
 \end{equation}
The inequality is not true for $0\leq \theta < \frac{1}2$ and $\tau>0$ since the left-hand side is positive and the right-hand side is negative. The inequality achieves equality when $\theta =\frac{1}2$ and in this case we must have $\nu=0$. Requiring $\theta > \frac{1}2$ and $\tau>0$ we see $2\theta-1 > 0. $ and $\tau +1 >0$, which allows division in \eqref{2nd3rd} to find
\begin{equation*}
    \frac{-1}{1+2\theta\tau} \leq \frac{1}{(1+2\theta)\tau},
\end{equation*}
which is clearly true since $1+2\theta \tau >0$ and $(1+2\theta)\tau >0$. Thus we impose the requirement $\theta \geq \frac{1}2$.
 Combine the results of three conditions,
 \begin{equation} \label{1stCombCond}
     \max\{-\frac{1+\tau}{\tau}, \frac{1-2\theta+\tau-2\theta\tau}{1+2\theta\tau}\}\leq \nu\leq \min\{1+\tau,\frac{-1+2\theta-\tau+2\theta\tau}{\tau+2\theta\tau}\}
 \end{equation}
 A simply calculation reveals that 
 \begin{equation*}
    - \frac{1+\tau}{\tau} < \frac{(1+\tau)(1-2\theta)}{1+2\theta\tau}
 \end{equation*}
 for $\theta, \tau > 0$. Thus the condition \eqref{1stCombCond} becomes
  \begin{equation} \label{2ndCombCond}
      \frac{1-2\theta+\tau-2\theta\tau}{1+2\theta\tau} \leq \nu\leq \min\{1+\tau,\frac{-1+2\theta-\tau+2\theta\tau}{\tau+2\theta\tau}\}.
 \end{equation}
 This restriction on $\nu$ is illustrated in figure \eqref{RestrictionOnNu} for several values of $\theta$. The region becomes larger as $\theta$ increases. 
 \begin{figure} [H]
\includegraphics[width=0.8\textwidth]{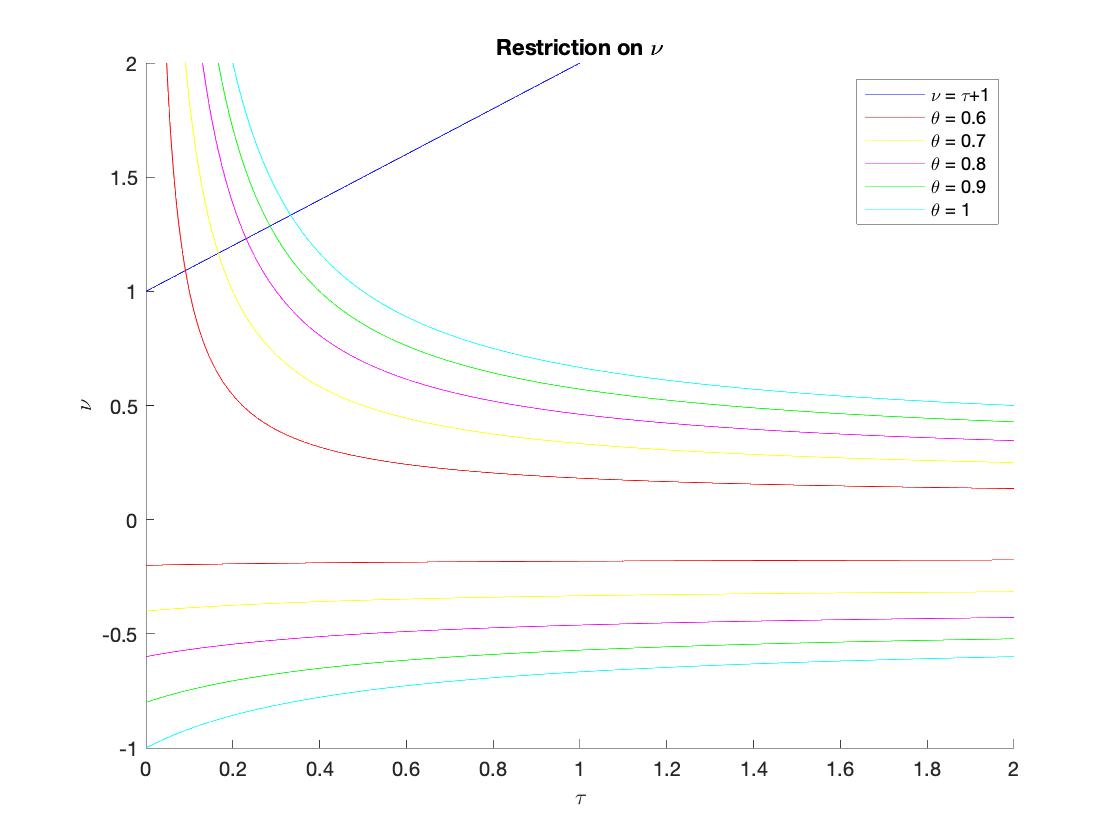}
\caption{Shows $\nu =1+\tau$ and $\nu = \frac{1-2\theta+\tau-2\theta\tau}{1+2\theta\tau}$ for different values of $\theta$, which are below the $\tau$-axis and $\nu = \frac{-1+2\theta-\tau+2\theta\tau}{\tau+2\theta\tau} $ for different values of $\theta$ which are above the $\tau$-axis. For $A$-stability, we must have the choice of $\nu$ and $\tau$ and $\theta$ be in the region below the blue line and in between two curves, determined by $\theta$, that are the same color.}
\label{RestrictionOnNu}
\end{figure}
\end{proof}
\begin{remark}
Another quick calculation can confirm that the $A$-stability restriction, \eqref{2ndCombCond}, on $\nu$ implies the $0$-stability restriction $-\frac{\tau+1}\tau < \nu < \frac{\tau+1}\tau$. For decreasing or constant timestep (i.e $0<\tau\leq0$) the method is second order and $A$-stable with the choice $\theta \geq \frac{1}2$ and $\nu = \frac{\tau (1+\tau)(2\theta-1)}{2\theta\tau+1}$. But for increasing timestep (i.e.$\tau>1$) the method cannot be both $A$-stable and 2nd order since $$ \frac{\tau (1+\tau)(2\theta-1)}{2\theta\tau+1} > \frac{(2\theta-1)(1+\tau)}{(2\theta+1)\tau}$$ in this case. 
\end{remark}
\section{Numerical Tests.}
\subsection{The Lorenz system}
Consider the Lorenz system \cite{lorenz1963deterministic}
\begin{equation*}
    \frac{dX}{dt} = 10(Y-X),\quad \frac{dY}{dt} = -XZ + 28X-Y, \quad \frac{dZ}{dt} = XY-\frac{8}{3}Z.
\end{equation*}
The initial conditions are $(X_0,Y_0,Z_0)= (0,1,0)$. The system is solved over the time interval $[0,5]$.
The reference solution is obtained by self-adaptive $RK45$. The results are shown in Figures \ref{fig1} and \ref{fig2}.

\begin{figure}[H]
    \centering
    \includegraphics[width=18cm]{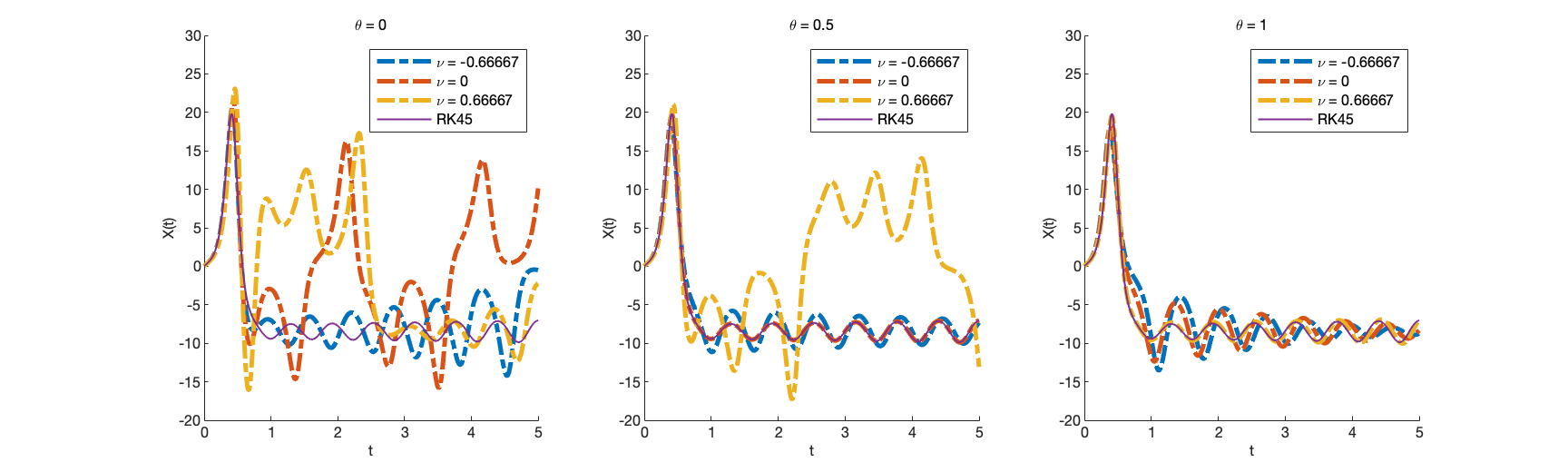}
    \caption{dt = 0.01}
    \label{fig1}
\end{figure}
\begin{figure}[H]
    \centering
    \includegraphics[width=18cm]{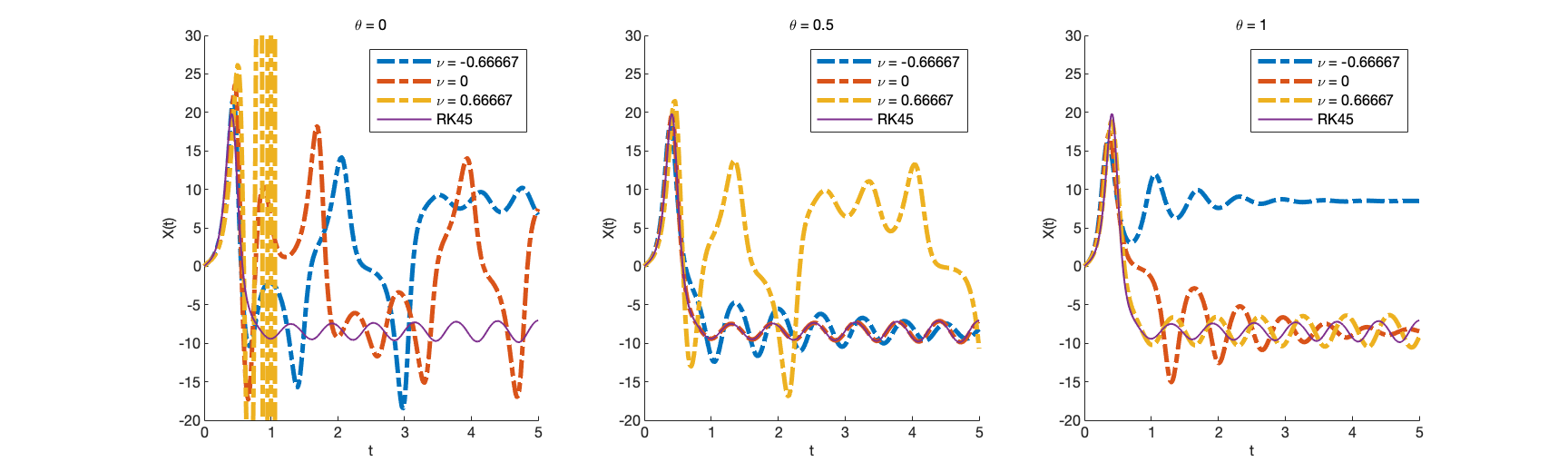}
    \caption{dt = 0.02}
    \label{fig2}
\end{figure}

\subsection{Periodic and quasi-periodic oscillations}
Consider the pendulum test problem \cite{li2015analysis,williams2013achieving} given by 
\begin{equation*}
    \frac{d\theta}{dt} = \frac{v}{L},\quad \frac{dv}{dt}=-g\sin{\theta}.
\end{equation*}
where $\theta$, $v$, $L$ and $g$ denote angular displacement, velocity along the arc, length of the pendulum and the acceleration due to gravity, respectively. Set $\theta(0)=0.9\pi$, $v(0) = 0$, $g=9.8$, time step $k =0.1$ and $L=49$.
The result are shown in Figures \ref{fig3} and \ref{fig4}.
\begin{figure}
    \centering
    \includegraphics[width=18cm]{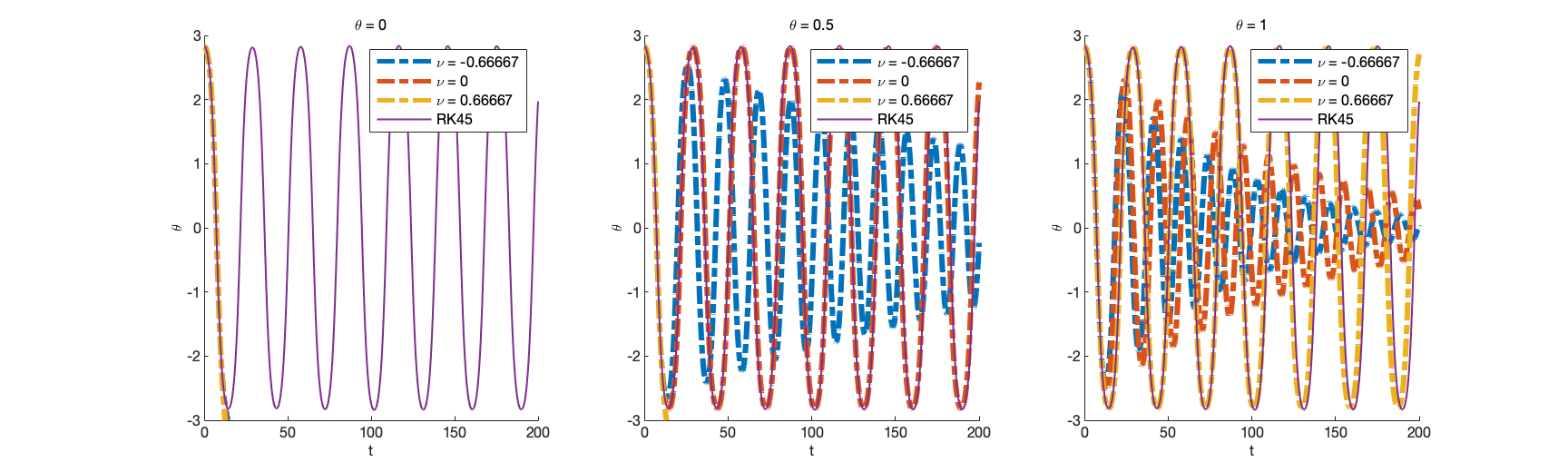}
    \caption{$\theta$ vs t}
    \label{fig3}
\end{figure}
\begin{figure}[H]
    \centering
    \includegraphics[width=18cm]{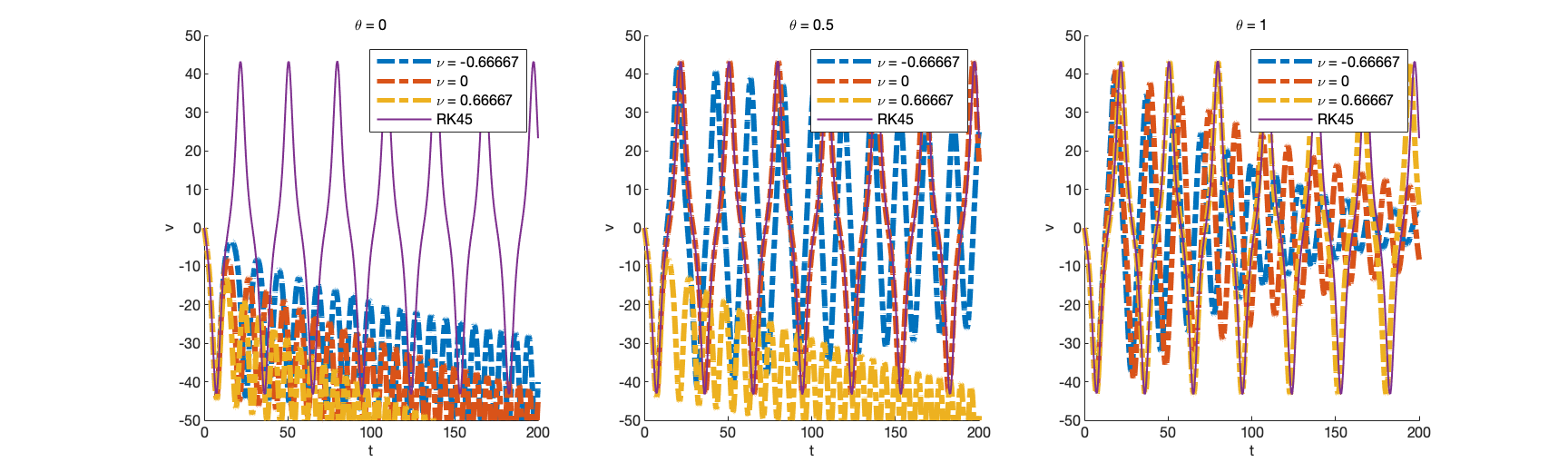}
    \caption{v vs t}
    \label{fig4}
\end{figure}
\subsection{Test problem with exact solution}
Consider the test problem
\begin{equation*}
    y' = \lambda(y-\sin{t}) +\cos{t}, \quad y(0)=1, \quad 0\geq t<1,
\end{equation*}
whose exact solution is $y(t) = e^{\lambda t} +\sin{t}$.
Consider using time step $k =0.01$ for $\lambda =10,1,0,-1,-10,-500$.
\begin{remark}
In Proposition 3.1, we showed second-order accuracy is attained when $\nu =(4\theta-2)/(2\theta+1) $.
    \begin{itemize}
    \item $\theta = 0$, $(4\theta-2)/(2\theta+1)= -2$,
    \item $\theta = 1/2$, $(4\theta-2)/(2\theta+1)= 0$,
    \item $\theta = 1$, $(4\theta-2)/(2\theta+1)= 2/3$.
    \end{itemize}
\end{remark}

\begin{figure}[H]
    \centering
    \begin{minipage}{\textwidth}
    \centering
    \includegraphics[width=18cm]{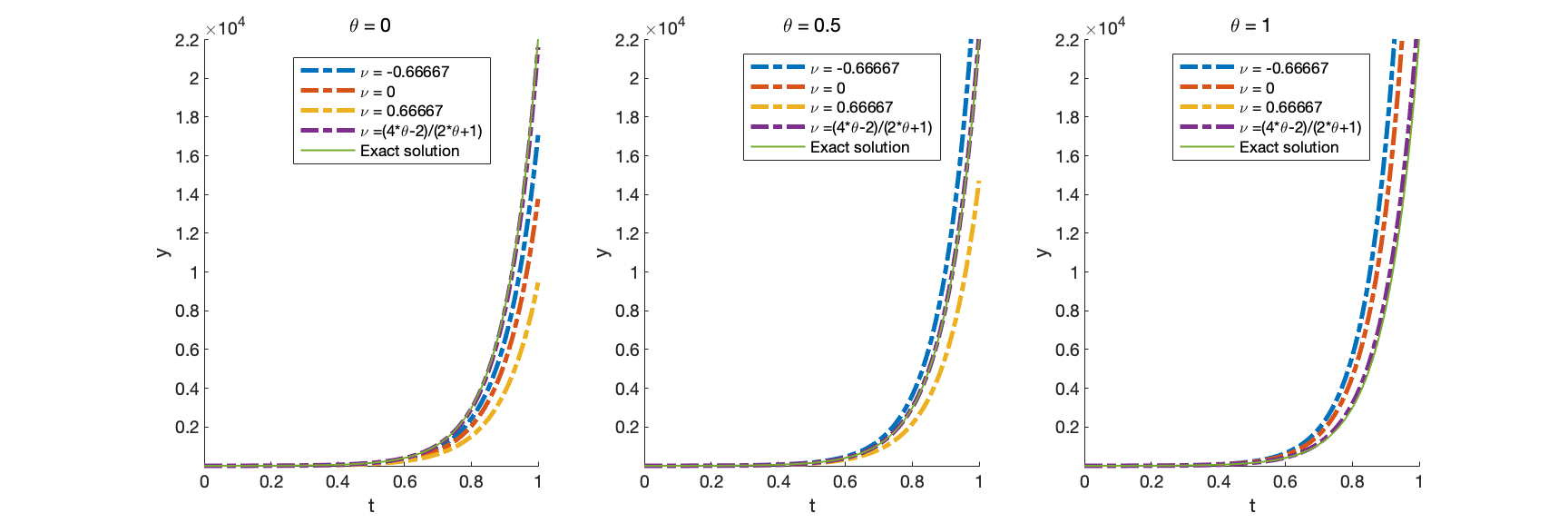}
    \caption{$\lambda = 10$}
    \end{minipage}
    \centering
    \begin{minipage}{\textwidth}
    \centering
    \includegraphics[width=18cm]{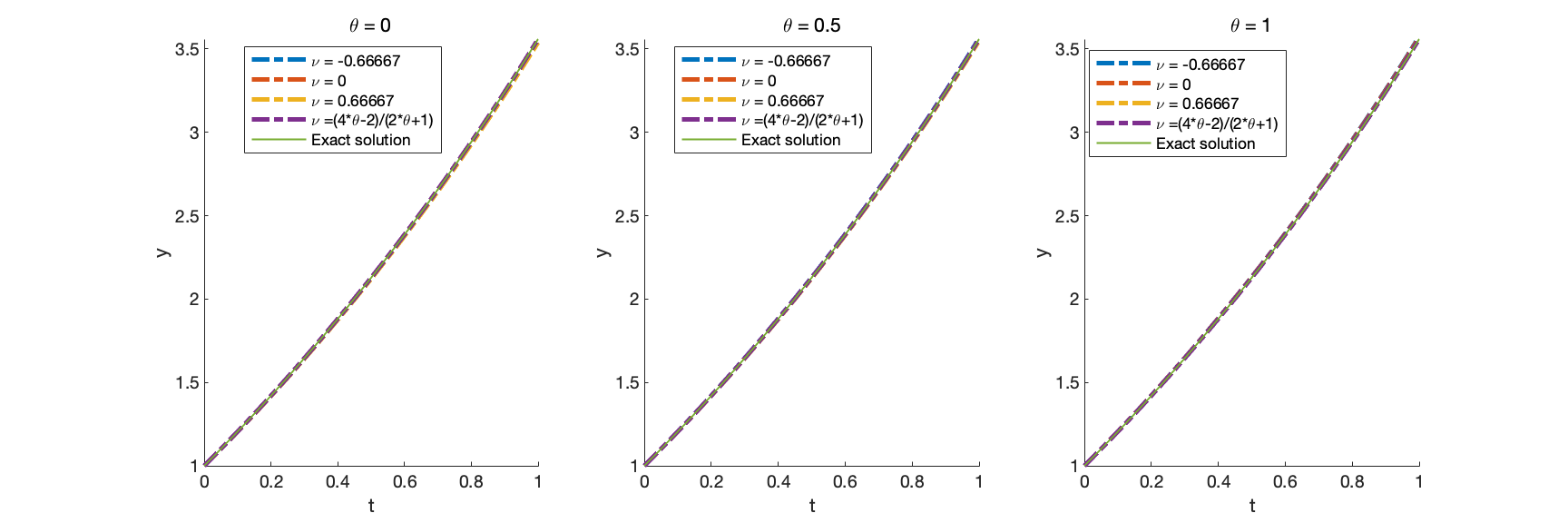}
    \caption{$\lambda = 1$}
    \end{minipage}
    \centering
    \begin{minipage}{\textwidth}
    \centering
    \includegraphics[width=18cm]{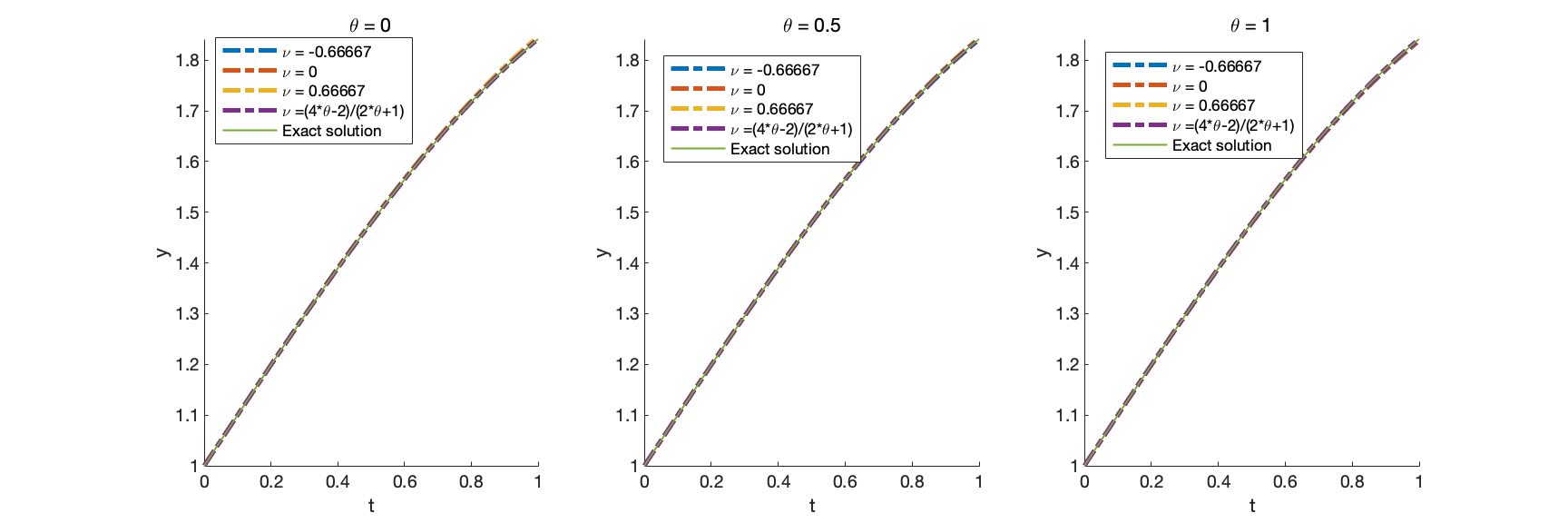}
    \caption{$\lambda = 0$}
    \end{minipage}
    \end{figure}
    \begin{figure}
    \centering
    \begin{minipage}{\textwidth}
    \centering
    \includegraphics[width=18cm]{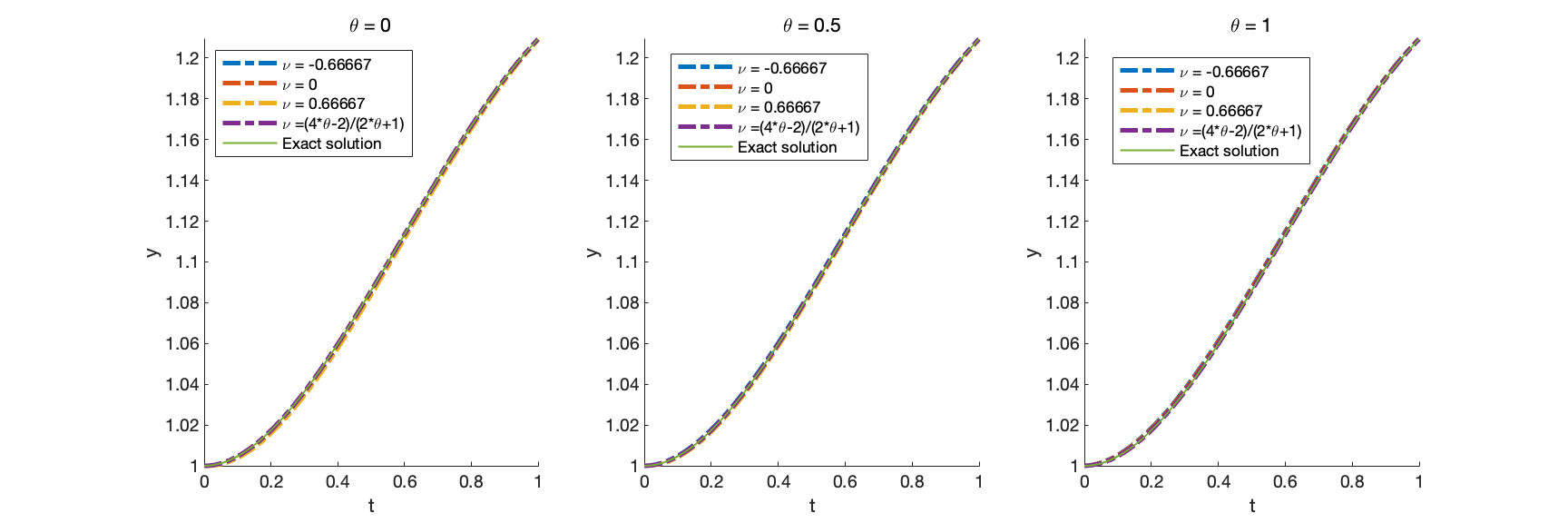}
    \caption{$\lambda = -1$}
    \end{minipage}
    \centering
    \begin{minipage}{\textwidth}
    \centering
    \includegraphics[width=18cm]{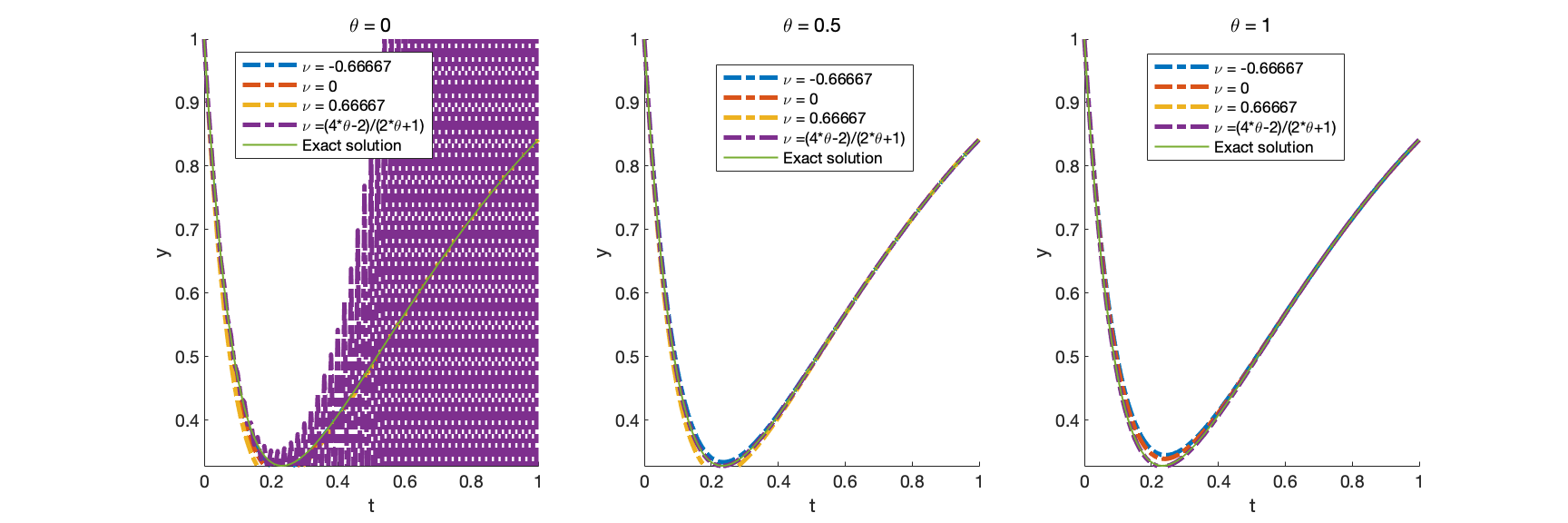}
    \caption{$\lambda = -10$}
    \end{minipage}
    \centering
    \begin{minipage}{\textwidth}
    \centering
    \includegraphics[width=18cm]{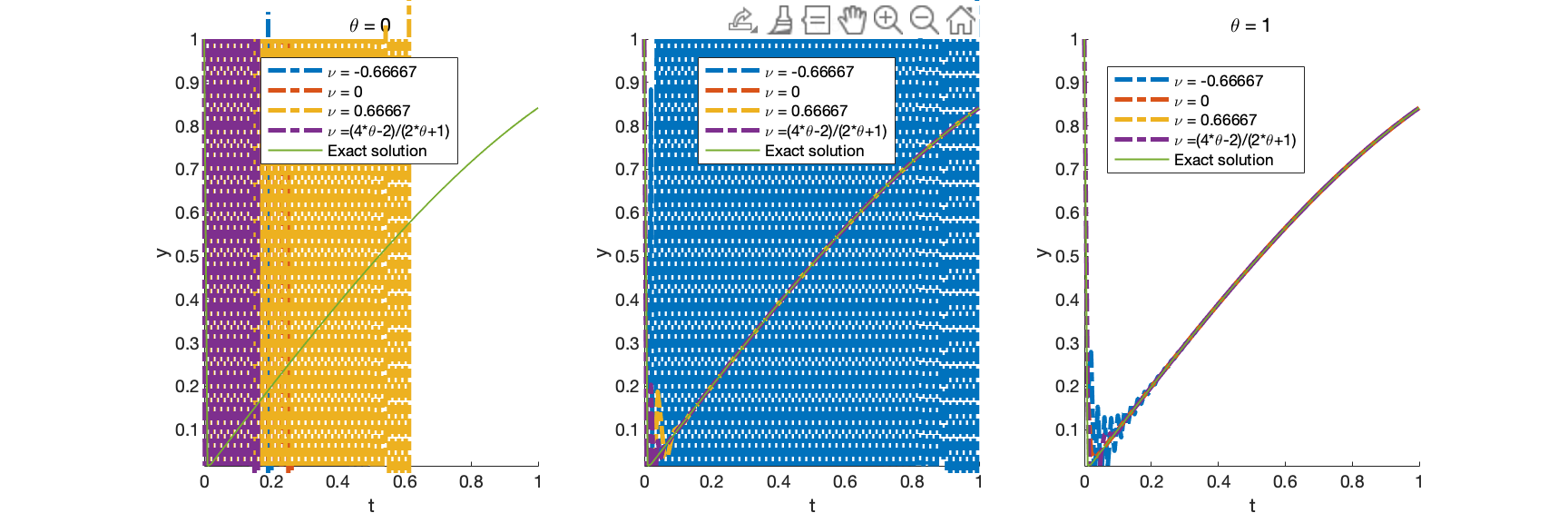}
    \caption{$\lambda = -500$}
    \end{minipage}
\end{figure}
\section{Convergence rate}
Consider the test problem
\begin{equation*}
    y' = \lambda(y-\sin{t}) +\cos{t}, \quad y(0)=1, \quad 0\geq t<1,
\end{equation*}
whose exact solution is $y(t) = e^{\lambda t} +\sin{t}$. Consider using time step $k =0.00125, 0.0025, 0.005, 0.01, 0.02$ for $\lambda =10,1,0,-1,-10$ to calculate the convergence rate. We got the same result for all cases of $\lambda$. Hence we showed for the case when $\lambda=-10$ in the following tables.
\begin{remark}
In Proposition \ref{p1c}, we get second-order accuracy when $\nu =(4\theta-2)/(2\theta+1) $.
\begin{itemize}
    \item $\theta = 0$, $(4\theta-2)/(2\theta+1)= -2$,
    \item $\theta = 1/2$, $(4\theta-2)/(2\theta+1)= 0$,
    \item $\theta = 1$, $(4\theta-2)/(2\theta+1)= 2/3$.
\end{itemize}
\end{remark}
\begin{table}[H]
\centering
\begin{tabular}{|l|l|l|l|l|l|l|l|l|}
\hline
Timestep & $\nu=-0.66667$ & Rate   & $\nu=0$ & Rate   & $\nu=0.66667$ & Rate   & $\nu=\frac{4\theta-2}{2\theta+1}$ & Rate   \\\hline
0.0013   & 4.9438e-04                  & -      & 9.8742e-04           & -      & 0.0020                     & -      & 0.1935                                                                                          & -      \\\hline
0.0025   & 9.9394e-04                  & 1.0075 & 0.0020               & 1.0054 & 0.0040                     & 1.0060 & 0.7781                                                                                          & 2.0074 \\\hline
0.0050   & 0.0020                      & 1.0158 & 0.0040               & 1.0109 & 0.0080                     & 1.0119 & 3.1372                                                                                          & 2.0115 \\\hline
0.0100   & 0.0041                      & 1.0348 & 0.0081               & 1.0225 & 0.0163                     & 1.0232 & 12.6357                                                                                         & 2.0100 \\\hline
0.0200   & 0.0087                      & 1.0849 & 0.0168               & 1.0477 & 0.0335                     & 1.0444 & 49.4689                                                                                         & 1.9690\\\hline
\end{tabular}
\caption{$\theta=0,\ \lambda=-10.$}
    \label{tab:tab1}
\end{table}
\begin{table}[H]
\centering
\begin{tabular}{|l|l|l|l|l|l|l|l|l|}
\hline
Timestep & $\nu=-0.66667$ & Rate   & $\nu=0$ & Rate   & $\nu=0.66667$ & Rate   & $\nu=\frac{4\theta-2}{2\theta+1}$ & Rate   \\\hline
0.0013   & 4.8942e-04                  & -      & 2.0649e-06           & -      & 9.8734e-04                   & -      & 2.0649e-06                                                                                         & -      \\\hline
0.0025   & 9.7398e-04                  & 0.9928 & 8.2597e-06               & 2.0001 & 0.0020                     & 1.0052 & 8.2597e-06                                                                                          & 2.0001 \\\hline
0.0050   & 0.0019                     & 0.9858 & 3.3044e-05              & 2.0002 & 0.0040                     & 1.0103 & 3.3044e-05                                                                                         & 2.0002 \\\hline
0.0100   & 0.0038                      & 0.9724 & 1.3226e-04               & 2.0009 & 0.0081                    & 1.0200 & 1.3226e-04                                                                                        & 2.0009 \\\hline
0.0200   & 0.0073                    & 0.9493 & 5.3042e-04              & 2.0037 & 0.0166                    & 1.0377 & 5.3042e-04                                                                                         & 2.0037\\\hline
\end{tabular}
\caption{$\theta=\frac{1}{2},\ \lambda=-10.$}
    \label{tab:tab2}
\end{table}

\begin{table}[H]
\centering
\begin{tabular}{|l|l|l|l|l|l|l|l|l|}
\hline
Timestep & $\nu=-0.66667$ & Rate   & $\nu=0$ & Rate   & $\nu=0.66667$ & Rate   & $\nu=\frac{4\theta-2}{2\theta+1}$ & Rate   \\\hline
0.0013   & 0.0015               & -      & 9.8017e-04           & -      & 1.8416e-05                  & -      & 1.8416e-05                                                                                       & -      \\\hline
0.0025   & 0.0029                  & 0.9945 & 0.0020               & 0.9948 & 7.2888e-05                    & 1.9847 & 7.2888e-05                                                                                         & 1.9847 \\\hline
0.0050   & 0.0058                    & 0.9892 & 0.0039             & 0.9897 & 2.8546e-04                     & 1.9695 & 2.8546e-04                                                                                       & 1.9695 \\\hline
0.0100   & 0.0115                     & 0.9786 & 0.0076               & 0.9798 & 0.0011                   & 1.9397 & 0.0011                                                                                        & 1.9397 \\\hline
0.0200   & 0.0223                    & 0.9581 & 0.0149             & 0.9615 & 0.0040                    & 1.8820 & 0.0040                                                                                          & 1.8820\\\hline
\end{tabular}
\caption{$\theta=1,\ \lambda=-10.$}
    \label{tab:tab3}
\end{table}
\section{Conclusions} Though the result for Backward Euler with time filter is known, we explored all possible choices of $\theta$ in our paper. We have shown that for different choices of $\nu$, we have different stability regions. We have shown that when $\theta>1/2$, we always have $A$-stability and when $\theta<1/2$, we have $A_0$-stability or $A_{\pi/4}$- stability.
\section*{Acknowledgement}
We want to thank our advisor Professor William J. Layton, for his insightful ideas and guidance throughout the research. We thank NSF for funding the project.
\bibliographystyle{plain}
\bibliography{references.bib}
\end{document}